\documentclass[11pt,a4paper]{article}
\usepackage{bigints}
\usepackage{enumerate}
\usepackage[utf8]{inputenc}
\usepackage[english]{babel}
\usepackage[margin=2.5cm]{geometry}

\usepackage{lipsum}

\usepackage{enumitem}

\setlist{nolistsep}
\usepackage{mathtools}
\mathtoolsset{showonlyrefs}
\usepackage{mathrsfs}
\usepackage{lettrine}
\usepackage{mfirstuc}
\usepackage{multicol}

\usepackage{bbm}
\usepackage{amsmath}
\usepackage{amssymb}
\usepackage{amsfonts}
\usepackage{amsthm}
\usepackage{wasysym}
\usepackage{dsfont}
\usepackage{graphicx}
\usepackage{wrapfig}
\usepackage[labelfont={bf,it},textfont=it]{caption}
\usepackage[hidelinks]{hyperref}
\usepackage{widetable}
\usepackage[scr=pxtx, scrscaled=1]{mathalfa}
\usepackage{xfrac}
\usepackage{soul}
\usepackage{tikz}
\usetikzlibrary{decorations.pathreplacing,calligraphy}
\usetikzlibrary{shapes.geometric}
\usepackage{subcaption}
\usepackage{setspace}
\usepackage{yfonts}
\usepackage{bm}

\newcommand{\abs}[1]{\lvert#1\rvert}
\newcommand{\norm}[1]{\lVert #1\rVert}

\renewcommand*{\P}{\mathbb{P}}
\newcommand*{\E}{\mathbb{E}}
\newcommand*{\R}{\mathbb{R}}

\newcommand*{\N}{\mathbb{N}}
\newcommand{\size}{F}
\newcommand{\presize}{F^N}

\newcommand{\Nt}{\lfloor Nt\rfloor}
\newcommand{\pii}{\tilde{\Pi}^{\text{\tiny{($i$)}}}}

\newcommand{\prepii}{\tilde{\Pi}^{\text{\tiny{($i$),$N$}}}}
\newcommand{\succi}{\stackrel{\text{\tiny{($i$)}}}{\succ}}
\newcommand{\simi}{\stackrel{\text{\tiny{($i,j$)}}}{\sim}}
\newcommand{\succsimi}{\stackrel{\text{\tiny{($i$)}}}{\succsim}}

\newcommand{\Succsim}[1]{\stackrel{\text{\tiny{(#1)}}}{\succsim}}

\newcommand{\demes}{D} 
\newcommand{\Ni}{\omega_iN} 
\renewcommand{\i}{\text{\tiny{$(i)$}}}
\renewcommand{\j}{\text{\tiny{$(j)$}}}
\newcommand{\loc}[1]{\text{\tiny{$(#1)$}}}
\renewcommand*{\Xi}{\varXi}
\renewcommand*{\epsilon}{\varepsilon}
\renewcommand*{\Theta}{\varTheta}

\renewcommand*{\Delta}{\varDelta}

\renewcommand{\vec}[1]{\underline{#1}}
\newcommand*{\1}{\mathds{1}}
\newcommand{\efflen}{\alpha}

\theoremstyle{plain}
\theoremstyle{definition}
\newtheorem{definition}{Definition}[section]

\newtheorem{lemma}[definition]{Lemma}

\newtheorem{proposition}[definition]{Proposition}
\theoremstyle{plain}
\newtheorem{theorem}{Theorem}
\theoremstyle{remark}
\newtheorem{remark}[definition]{Remark}

\usepackage{textcase}
\usepackage{xcolor}

\title{Multi-type $\Xi$-coalescents from structured population models with bottlenecks}

\author{Marta Dai Pra \thanks{Institut f\"ur Mathematik, Humboldt-Universit\"at zu Berlin, Germany, {\tt daimarta@hu-berlin.de}} \and Alison Etheridge \thanks{Department of Statistics, University of Oxford, United Kingdom, {\tt etheridg@stats.ox.ac.uk}} \and Jere Koskela \thanks{School of Mathematics, Statistics and Physics, Newcastle University, and Department of Statistics, University of Warwick, United Kingdom, {\tt jere.koskela@newcastle.ac.uk}}
\and Maite Wilke-Berenguer \thanks{Institut f\"ur Mathematik, Humboldt-Universit\"at zu Berlin, Germany {\tt maite.wilkeberenguer@hu-berlin.de} }}
\date{}

\begin{document}

\maketitle

\begin{abstract}
	We introduce an individual-based model for structured populations undergoing demographic bottlenecks, i.e.\ drastic reductions in population size that last many generations and can have arbitrary shapes. We first show that the (non-Markovian) allele-frequency process converges to a Markovian diffusion process with jumps in a suitable relaxation of the Skorokhod J1 topology. Backward in time we find that genealogies of samples of individuals are described by multi-type $\Xi$-coalescents presenting multiple simultaneous mergers with simultaneous migrations. These coalescents are also moment-duals of the limiting jump diffusions. We then show through a numerical study that our model is flexible and can predict various shapes for the site frequency spectrum, consistent with real data, using a small number of interpretable parameters.
    \medbreak
\end{abstract}

\section{Introduction}
Most mathematical theory and statistical tools in population genetics are based on the Kingman coalescent \cite{Kingman_1982}, which is a random tree describing common ancestry of genetic samples from large populations. Its evolution consists of binary mergers that happen at rate $1$ for each pair of lineages. A key reason for its ubiquity is robustness: the Kingman coalescent arises as a scaling limit from a broad class of individual-based models \cite{Ca74,Mohle98}. On the other hand, one of its key limitations is that the number of offspring per individual has to be small in comparison to the population size. This is often a reasonable assumption provided that the population size remains constantly large. However, many organisms exhibit recurrent strong variation in population sizes \cite{TL_2014} which can occur, for example, due to limited resources or natural catastrophes. Surviving organisms can then repopulate the environment in a relatively short time, due for example to reduced competition, causing many individuals to descend from few ancestors.

Such genealogies are better described by coalescents with multiple mergers, also called $\Lambda$-coalescents \cite{DK99,Sag99,Pit99}, or with simultaneous multiple mergers, known as $\Xi$-coalescents \cite{Schw_2000,MS01}. In addition to a large variation in population size, other effects giving rise to multiple merger coalescents include skewed offspring production, typical of marine species, \cite{Bjarki16}, local extinction-recolonization events in spatially structured populations \cite{TV09}, natural selection \cite{Sch17}, and range expansion \cite{BHK21}.

In this paper, we consider demographic bottlenecks, i.e.\ severe declines in population size that can last many generations, in a spatially-structured population.
In nature, these declines can be caused, for example, by natural disasters such as fires or floods. The reasons for modelling populations undergoing catastrophic events range from inferring human population history \cite{SBB,LD11} to reconstructing the evolutionary history of certain pathogens \cite{MGF21}. Earlier works on population bottlenecks such as \cite{BBMST_2008, GMS_2022} show that genealogies from individual-based models presenting drastic fluctuation in population size can converge to multiple merger coalescents in the infinite-population limit. 

Our aim is to add a geographical structure to this picture. Real examples of bottlenecks affecting structured populations can be found in marine species: one example is the Atlantic cod, which is widespread in various separated areas of the Icelandic sea between which migration is possible \cite{OWEP14}. 
In this scenario, we can think of a bottleneck as a fishing period in a specific area. Another possible application may be the spreading of certain pathogens between hosts where bottlenecks arise due to the effect of the host immune system or of antibiotic treatments \cite{MGF21}.

We build an individual-based model of a population divided into a finite number of islands affected by recurrent bottlenecks. We assume that bottlenecks affect one island at a time. 
Furthermore, we also generalise the common assumption that population size during a bottleneck is constant \cite{BBMST_2008,GMS_2022}. 
Instead, our analysis allows an essentially arbitrary population size trajectory.

We are going to consider scaling limits in which the total population size, $N$, goes to infinity.
Following the idea of \cite{GMS_2022}, we consider bottlenecks of two regimes: we say that a bottleneck is \emph{drastic} if, for its entire duration, the size of the affected island remains of order one and is independent of $N$; we say that a bottleneck is \emph{soft} if the size of the affected island grows to infinity, but is $o(N)$.

For both models, we investigate how the presence of bottlenecks affects the allele-frequency process and show that it converges to a diffusion process with jumps.

We then study the behaviour of the ancestral lines backward in time. 
We prove that the limiting genealogy is a multi-type $\Xi$-coalescent presenting multiple simultaneous mergers and simultaneous migrations (see Figure \ref{fig:coalescent}). Multi-type coalescents have been of great interest in recent times \cite{Eld09,JKR23,Mohle23}, and, in our setting, types stand for geographical locations. The peculiarity of our model is that it leads to a coalescent featuring multiple mergers and changes of types (migrations) simultaneously. 
This behaviour was already present in previous models \cite{JKR23}, but our novelty is that we obtain it as the scaling limit of a concrete individual-based model. Simultaneous migrations and coalescences arise from the fact that, during a bottleneck, we still have migrants leaving from and coming into the affected island.

\begin{figure}
\centering
\includegraphics[height=6cm]{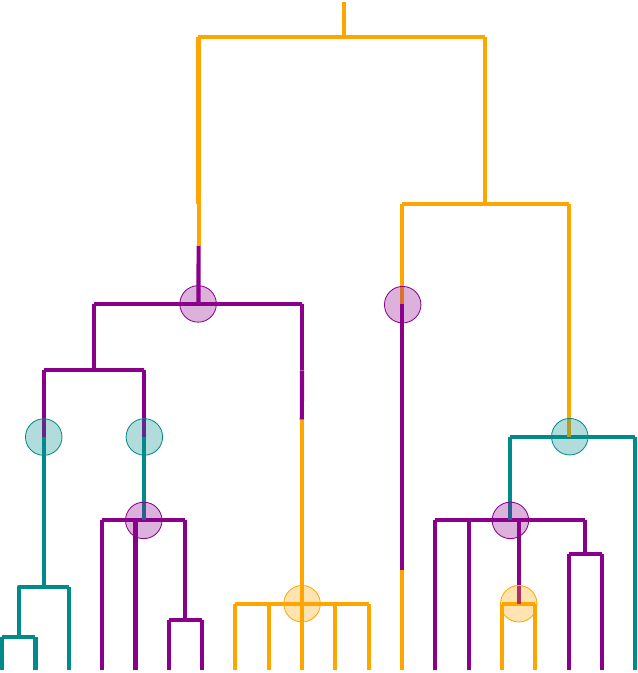}
\caption{An example of a coalescent tree obtained as a limiting genealogy of our model. Different colours represent different demes, in this case they are $D=3$. We note that we can have binary mergers and single migrations as well as simultaneous events (multiple simultaneous mergers and simultaneous migrations). The latter, caused by bottlenecks, are highlighted with circles. Circles on the same level show all the events happening during a particular bottleneck and their colour corresponds to the affected deme.} \label{fig:coalescent}
\end{figure}

\subsection{The model}\label{sec:model}

We study a \emph{structured} Wright--Fisher model with \emph{bottlenecks}. More precisely, we consider a haploid population evolving in discrete generations $g\in\N_0$, where $\N_0=\{0,1,2,\ldots \}$. We assume the population is divided into $\demes \in \N$ \emph{demes} with initial population sizes \mbox {$N_i=\Ni\in \N$}, with $\omega_i > 0$ and $\sum_{i=1}^D \omega_i=1$, such that the total population size is equal to $N=\sum_{i=1}^\demes \Ni$. One can think of demes as separate \emph{islands}. 

The population is affected by recurring bottlenecks, i.e.\ exogenous events that substantially reduce the population size for one or several generations. These events affect exactly one island at a time and, at the end of a bottleneck, the population size on the affected island is restored to its usual value; each bottleneck can have a different deterministic shape (see Figure \ref{fig:pop-sizes}). 

Fix a probability space $(\Omega,\mathcal{F},\P)$. We enumerate consecutive bottlenecks with an index $k \in \N$ and consider the following sequences of independent random variables to describe their properties:
\begin{itemize}
\item  $\{l_k^N\}_{k\in \N}$ are the lengths of the bottlenecks;
\item $\{s_k^N\}_{k\in \N}$ are the times between two bottlenecks for $k\geq 2$, i.e.\ $s_k^N$ is the number of generations from the end of the $(k-1)$-th bottleneck to the beginning of the $k$-th bottleneck, while $s_1^N$ is the beginning of the first bottleneck;
\item $\{\beta_k\}_{k\in \N}$ are i.i.d.\ uniform random variables on $\{1,\ldots,\demes\}$ determining the bottlenecks' locations.
\end{itemize}
When the \mbox{$k$-th} bottleneck occurs, if $\beta_k=i$, it affects the $i$-th subpopulation and it lasts for $l_k^N$ generations. The size of the affected deme during the \mbox{$k$-th} bottleneck is determined by a function $\size_k^N:\N_0\rightarrow \N$. The random sequence of the $i$-th subpopulation size is given by
\begin{equation}\label{eq:pop-size}
 R^N_{i}(g) =   \left\{ \begin{array}{ll}
 \size_k^N(g-\underline{t}^N_{k}-1),  & \,\,\, \text{if} \,\,\underline{t}_k^N< g \leq \underline{t}_k^N+l^N_k\,\,	\text{and} \,\, \beta_k=i,\\
\Ni, & \,\,\, \text{otherwise}, 
 \end{array}
 \right.			
 \end{equation}
where $\underline{t}_k^N:=\sum_{m=1}^{k-1}(s_m^N+l_m^N)+s_{k}^N$ is the random generation before the beginning of the $k$-th bottleneck. We also denote by $\overline{t}^N_k=\sum_{m=1}^{k}(s_m^N+l_m^N)$ the last generation of the $k$-th bottleneck.
For now, we do not make any further assumption on the distributions of these random variables, as they will be specified in the following sections.

The evolution of the population at each generation $g \in \N_0$ consists of two steps:
 
\begin{enumerate}
\item \emph{Reproduction:} Reproduction takes place within each deme and consists of a Wright--Fisher update, i.e.\ each individual chooses one parent uniformly from the previous generation and independently from the other individuals in its generation.
\item \emph{Migration:}	A fixed number $c_{ij}$ of individuals, chosen uniformly without replacement, migrates from island $i$ to island $j$. We assume
\begin{equation*}
c_i:=\sum_{j\neq i}c_{ij}=\sum_{j\neq i}c_{ji},
\end{equation*}
so that the size of each subpopulation is maintained under migration; $c_i$ is the total number of individuals leaving and coming into deme $i$, where we assume that for every $k \in \N$
\begin{equation*}
\max_i\{c_i\}\leq \size^N_k(g)\leq \min_i\{N_i\},
\end{equation*}
for every $g\in \N_0$. We are assuming the number of migrants per generation is $O(1)$, and independent of whether a bottleneck is taking place. This is a strong assumption for drastic bottlenecks, where the finite population size on a deme during a bottleneck is sensitive to exact migrant numbers. For soft bottlenecks, small changes in migration patterns will not be relevant for the large-population scaling limit. However, we would like to mention that the model can be extended to the case where the number of individuals leaving a deme decreases during a bottleneck.
\end{enumerate}

\begin{figure}
\centering
\includegraphics[width=.4\textwidth,height=.7\textwidth]{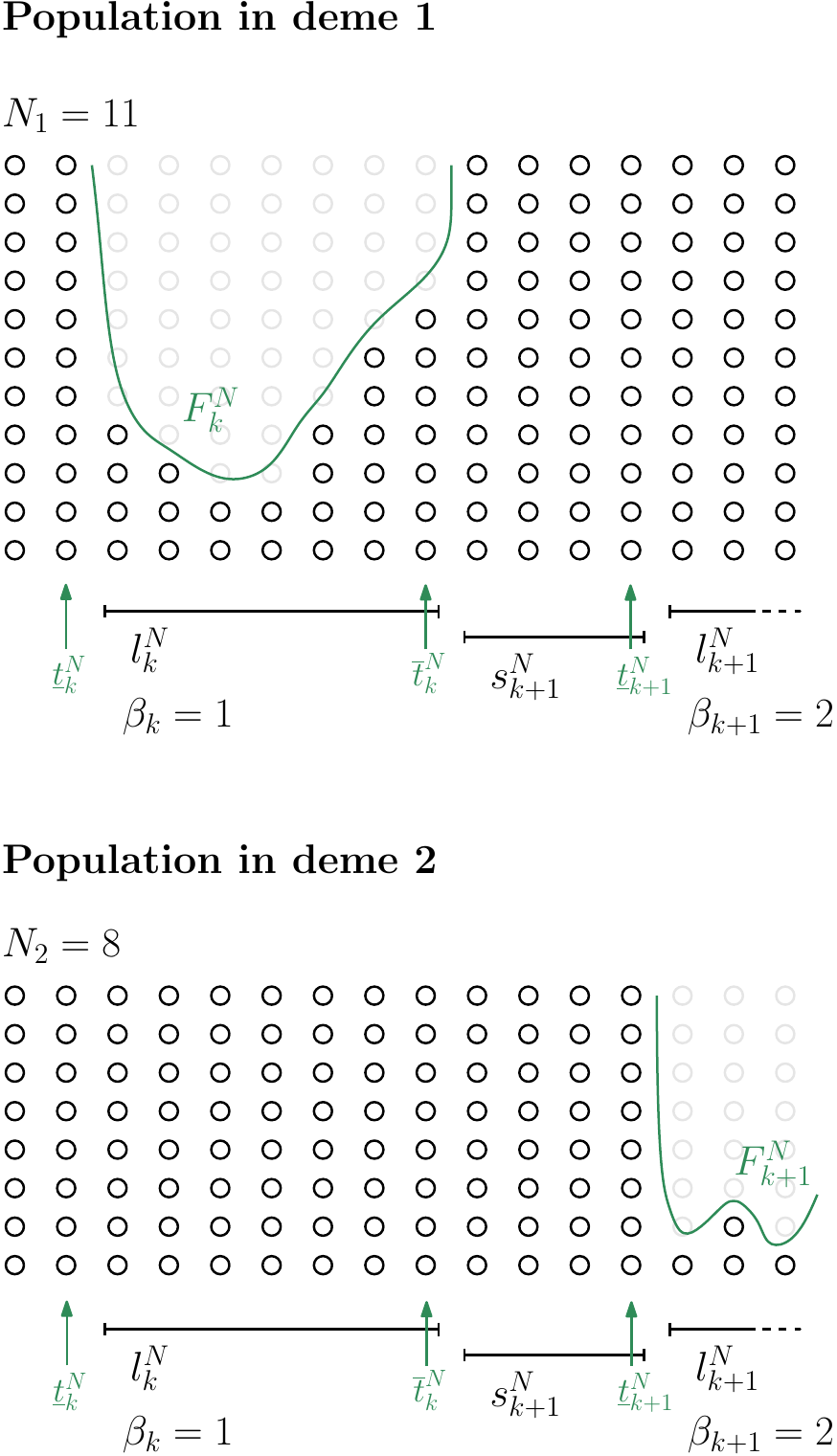}
\caption{Example of the effect of bottlenecks on the population size in the case D=2. The circles represent individuals, and their number is constant in between bottlenecks and decreases in a deme when a bottleneck happens there. See the main text for a full explanation of the notation.} \label{fig:pop-sizes}
\end{figure}

There are two types of processes we study from these population models, one forward in time and one backward in time. Forward in time, we assume that each individual carries an allele, from a bi-allelic space $E=\{a,A\}$, which it passes on to its offspring, and we study the allele-frequency process. Backward in time, we study the coalescent process that describes the genealogy of a sample from  the population.

In Section \ref{section:drastic}, we examine \emph{drastic} bottlenecks, i.e.\ catastrophes during which the size of the affected island remains bounded as $N \to \infty$. We will observe that the length of the bottlenecks must remain bounded in order to get a non-trivial limit, and we prove that the corresponding $D$-dimensional allele-frequency process converges to a structured Wright--Fisher diffusion with jumps. Indeed if the length of a bottleneck would grow to infinity, the frequency of one allele on the affected island would a.s. be in $\{0,1\}$ by the end. We then find a dual process of this limit, which is the block-counting process of a multi-type $\Xi$-coalescent presenting multiple mergers and migrations simultaneously.

In Section \ref{section:soft}, we consider \emph{soft} bottlenecks during which the population size grows with $N$, but at a lower order. More precisely, we take bottleneck lengths and population sizes to satisfy, for every $k \in \N$,
\begin{equation*}
\lim_{N \to \infty} l^N_k= \infty, \quad \lim_{N \to \infty} \frac{l^N_k}{N}=0,	
\quad \lim_{N \to \infty}\frac{\presize_k(\lfloor l_k^Nt\rfloor)}{l_k^N}=\varphi(t) \hspace{2em}\text{in probability, } \forall t\in[0,1].
\end{equation*}
We first study the backward-in-time evolution and prove that the genealogy converges to a (different) multi-type $\Xi$-coalescent presenting multiple mergers and migrations simultaneously. We then prove a convergence result for the forward-in-time allele-frequency processes.
We call the models in Section \ref{section:drastic} (resp.\ \ref{section:soft}) structured Wright--Fisher model with drastic (resp.\ soft) bottlenecks.

Finally, in Section \ref{section:simulations}, we show some simulations of the Site Frequency Spectrum (SFS) of the previous models. We see the flexibility of our model that which predict very different SFS shapes using a small number of interpretable parameters.
\begin{remark}
When $N$ is fixed, the population size on the deme affected by a bottleneck can decline from and grow back to its regular size gradually over a number of generations. As $N \to \infty$, these phases of decline and regrowth necessarily become sharper as the regular population size grows. However, they can still take place over a number of generations. A drastic bottleneck lasts only finitely many generations and hence the change in population size will necessarily resemble a Heaviside function on the $N$-generation diffusive timescale as $N \to \infty$. Soft bottlenecks allow for less extreme variation in population size on a longer timescale, for example a bottleneck of duration $l^N \gg \log(N)$ could reduce the population to a constant level $l^N$ for $l^N-\log(N)$ generations and, in the last $\log(N)$ generations, the population grows back exponentially. This example could be related to the example of fishing periods if we imagine that, as soon as fishing ends, the reduced population grows back exponentially fast.
\end{remark}

\section{Drastic bottlenecks}\label{section:drastic}
In this section, we study the impact of \emph{drastic} bottlenecks on infinite-population scaling limits. At first, we investigate the allele-frequency process and prove its convergence, after a suitable time rescaling, to a diffusion process with jumps. In the second part of the section, we investigate the behaviour of the genealogy backward in time. We prove that the limiting frequency process is the moment dual of the block-counting process of a coalescent presenting multiple simultaneous mergers and migrations. 
 
We retain the notation of the previous section. Let $\gamma>0$. We make the following assumptions:
\begin{enumerate}[label=\textbf{D.\arabic*}]
\item \label{enumerate:1} Let $L$ be a probability measure on $\N$ and $\{l_k^N\}_{k\in \N}$ a sequence of i.i.d.\ random variables converging in distribution to $L$ as $N$ tends to infinity.
\item \label{enumerate:2} Let $\{s_k^N\}_{k\in \N}$ be a sequence of i.i.d.\ random variables following a geometric distribution with parameter $\gamma/N$.
\item \label{enumerate:3}: Let $\size:\N \rightarrow \N$ be a function such that 
\begin{equation*}
\max_i\{c_i\}\leq \size(g)\leq \min_i\{N_i\},
\end{equation*}
for every $g\in \N$.
\end{enumerate}
When the \mbox{$k$-th} bottleneck occurs, the size of the affected island is reduced to $\size(g-\underline{t}^N_{k})$ for $g \in \{\underline{t}_k^N+1,\ldots,\underline{t}_k^N+l^N_k\}$, thus we have that $\size_k^N(g)=\size(g+1)$ (see \eqref{eq:pop-size}). We start enumerating the bottleneck generations from $1$ and not $0$ to simplify later computations, we will go back to the former choice in Section~\ref{section:soft}.  
  \begin{remark}
Since $L$ does not depend on $N$, when we rescale time by a factor $N$ and consider the limit as $N \rightarrow \infty$, the duration of the bottlenecks is negligible. Nevertheless, since the population size on the affected island stays bounded during that time, the effect of these events remains visible in the limit, causing the jumps in the limiting allele-frequency process.
\end{remark}

\subsection{The frequency process}
We assume that each individual carries one of two possible alleles from the genetic type space $E=\{a,A\}$ and that each individual inherits its parent's type.
In this section, we investigate the behaviour of the proportion of individuals carrying allele $a$.
Given the initial population size $N$ and the initial type configurations, we consider the (non-Markovian) process $\mathbf{X}^N(g)=(X^{N}_1(g),\ldots,X^{N}_{\demes}(g))_{g\in \N_0}$ where $X^{N}_i(g)$ denotes the proportion of $a$ individuals in the $i$-th deme at time $g$.

\begin{remark}\label{remark:J1-top}
Since the size of a deme affected by a bottleneck is bounded, the frequency process can have big fluctuations throughout its duration. Hence we cannot obtain convergence to a jump-diffusion process in the Skorokhod J1 topology. As our aim is to still obtain a large population limit, we present a multi-dimensional version of the metric introduced in \cite[Section 3]{GMS_2022}. As the authors underline, the points that prevent the usual convergence are just the ones inside bottlenecks, where we observe large fluctuations.  However, this new metric will allow functions to be ``distant'' on a negligible set, and the set of points inside bottlenecks has Lebesgue measure $0$ when the time is rescaled by $N$ and $N \rightarrow \infty$.

For $T>0$, let $\mathcal{D}([0,T];\R^\demes)$ be the space of $\R^\demes$-valued c\`adl\`ag functions on $[0,T]$. Define also the set $\mathcal{I}$ of finite unions of half-open intervals

\begin{equation*}
\mathcal{I} = \{\cup_{i=1}^n[a_i,b_i):n\in\N,\, 0\leq a_i<b_i\leq T\},
\end{equation*}
and the set 
\begin{equation*}
\mathcal{F} = \{f:[0,T]\to[0,T]: \, f(0)=0,\, f(T)=T,\, f \text{ is continuous and strictly increasing}\}.
\end{equation*}

\begin{definition}\label{def:dmu}
Let $\lambda$ be the Lebesgue measure on $[0,T]$  and consider the identity function Id on $[0,T]$. For any $x_1,x_2\in \mathcal{D}([0,T];\R^\demes)$, we define the function $\overline{d}_{\lambda}$	 as
\begin{equation*}
\overline{d}_{\lambda}(x_1,x_2) = \inf_{A\in \mathcal{I},f\in \mathcal{F}}\{\norm{\1_A(x_1-x_2 \circ f)}_{\infty}\lor \norm{\text{Id}-f}_{\infty}\lor \lambda([0,T]\setminus A)\lor  \norm{x_1(T)-x_2(T)}\},	
\end{equation*}
where $\norm{\cdot}_{\infty}$ denotes the uniform norm and $\norm{\cdot}$ the Euclidean norm on $\R^\demes$.
\end{definition}
One can prove that $\overline{d}_\lambda(x_1,x_2)=0$ if and only if $x_1=x_2$ and that $\overline{d}_\lambda$ satisfies the triangle inequality following the steps of Proposition 3.2 in \cite{GMS_2022}. We then define the metric $d_\lambda$ on $\mathcal{D}([0,T];\R^\demes)$ as
\begin{equation*}
d_{\lambda}(x_1,x_2)=\frac{\overline{d}_{\lambda}(x_1,x_2)+\overline{d}_{\lambda}(x_2,x_1)}{2}.
\end{equation*}

The metric $d_\lambda$ is a slight modification of the Skorokhod metric on $\mathcal{D}([0,T];\R^\demes)$. From Proposition 3.2 of \cite{GMS_2022} it follows that convergence according to $d_\lambda$ implies convergence in measure on the Skorokhod space. It is also true that the usual convergence in the J1 topology implies convergence in $d_\lambda$.
\end{remark}

Our aim is to prove that the frequency process $\{\mathbf{X}^{N}(g)\}_{g\in \N_0}$ converges weakly, in the topology induced by $d_{\lambda}$, to a diffusion process with a \emph{Wright--Fisher} component and a \emph{jump} component. In order to describe the jumps in the limit, we need to understand the evolution of the frequency during a prototypical bottleneck. To this end, fix $i\in \{1,\ldots,\demes\}$ and $\tau \in \N$. We want to understand how the allele frequency changes during a bottleneck affecting deme $i$ and lasting $\tau$ generations. 
The first bottleneck generation consists of $\size(1)$ individuals, of which $\size(1)-c_{i}$ are descendants of parents on deme $i$, and $c_{i}=\sum_{j\neq i}c_{ji}$ are migrants from the other demes. 
At each following time step, $c_i$ new migrants arrive, hence the total number of families created during the complete duration of the bottleneck is equal to $\size(1)+c_i\, (\tau-1)$. 
\begin{remark}
Of the $c_i$ families that arrive at each step of the bottleneck in the affected deme $i$, some of them can be families that were in deme $i$ before. 
However, this would require a member of a specific family that previously migrated from deme $i$ to go back, and this happens with probability $O(1/N)$ so it is a negligible event in the bottleneck time scale.
\end{remark}

We now construct a Markov chain $\{\mathbf{A}^\i(r)\}_{r\in \N}=\{(\mathbf{A}^\i_1(r),\ldots,\mathbf{A}^\i_\demes(r)\}_{r\in \N}$  on $\N_0^{\demes \times \N}$ in such a way that, at time $r=\tau$, i.e.\  the end of the bottleneck, $\mathbf{A}^\i(\tau)$ \label{object:familysizesendofbottleneck} gives us the random number of descendants of all the families involved in the bottleneck, sorted by their deme of origin (see Figure \ref{fig:drastic-frequency}).
For every $j=\{1,\ldots,\demes\}$ and $r\in \N$, we denote by 
\begin{equation*}
\mathbf{A}^\i_j(r)=(a^\i_{jm}(r))_{m\in \N}
\end{equation*}
the sizes of the families generated by individuals originally from deme $j$, starting with
\begin{align*}
\mathbf{A}^\i_i(1)&=(\underbrace{1,\ldots,1}_{\text{$\size(1)-c_i$}},0,\ldots),\\
\mathbf{A}^\i_j(1)&=(\underbrace{1,\ldots,1}_{\text{$c_{ji}$}},0,\ldots) \hspace{0.6cm} \forall j\neq i.
\end{align*}
Notice that at time $r$, $A^\i_j(r)$ has at most $c_{ji}\cdot r$ non-zero entries if $j \neq i$, and $\size(1)-c_i$ if $j=i$. For every $r \in \N$, we have
\begin{equation*}
\sum_{j,m}a^\i_{jm}(r)=\size(r),
\end{equation*}
and we define the relative family sizes as 
\begin{equation}\label{eq:relative-size}
p^\i_{jm}(r):=a^\i_{jm}(r)/\size(r).
\end{equation}
Following the update mechanism described in Section \ref{sec:model}, in generation $r+1$, $\size(r+1)-c_i$ individuals choose a parent belonging to family $jm$ with probability $p^\i_{jm}(r)$, and $c_{ji}$ new migrants arrive from deme $j$ each creating a new family of size $1$.
Hence we have
\begin{align*}
a^\i_{jm}(r+1)&= \text{number of offspring of family $jm$} \hspace{0.5cm} \text{if } j\neq i \text{ and } m\in \{1,\ldots,c_{ji}\cdot r\}, \\
a^\i_{jm}(r+1)&= 1 \hspace{0.5cm} \text{if } j\neq i \text{ and } m\in \{c_{ji}\cdot r+1,\ldots,c_{ji}\cdot(r+1)\}, \\
a^\i_{im}(r+1)&= \text{number of offspring of family $im$} \hspace{0.5cm} \text{if } m\in \{1,\ldots,\size(1)-c_i)\}.
\end{align*}

\begin{figure}[h]
\begin{center}
\includegraphics[scale=0.75]{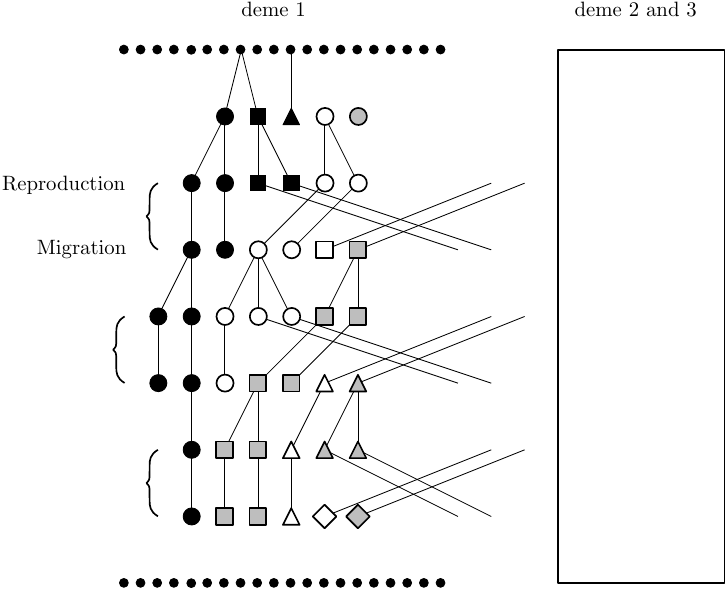}
	
\end{center}
\caption{In this picture we show the evolution of the population with $\demes=3$ when a bottleneck affects deme $1$ for $\tau=4$ generations and with $c_1=2$. At each step of the bottleneck two individuals come from the other demes: one from the second (white) and one from the third (grey). In $\size(1)+c_1(\tau-1)=11$ different colours and shapes we highlight the families whose size we want to know at the end of the bottleneck. Their sizes at the end of the bottleneck are $\mathbf{A}^1_1(\tau)=(1,0,0),\; \mathbf{A}^1_2(\tau)=(0,0,1,1),\; \mathbf{A}^1_3(\tau)=(0,2,0,1)$. }\label{fig:drastic-frequency}
\end{figure}
\begin{remark}
As the sizes of the demes unaffected by the bottleneck remain of order $N$, as $N \to \infty$, and the bottleneck's duration is negligible in the scaling limit, the frequency there will remain constant up to an error of order $1/N$. For this reason, we assume it to be constant even in the pre-limiting model. 
\end{remark}

Once we have these family sizes at the end of the bottleneck, we are able to write down the final proportion of type $a$ individuals. Assuming that the frequency right before the bottleneck was $\mathbf{x}=\{x_1,\ldots ,x_D\}\in [0,1]^D$, in the first step, the $\size(1)-c_i$ individuals independently choose an $a$ parent with probability $x_i$, and the other $c_i=\sum_{j\neq i}c_{ji}$ with probability $x_j$, depending on which deme $j$ is the origin of each of their ancestor. At each following step, new migrants from demes $j$ for $j \neq i$ generate $c_{ji}$ new families which are of type $a$ with probability $x_j$. Therefore, once the founder of a family $jm$ chooses a type, $a^\i_{jm}(\tau)$ descendants have the same type at the end of the bottleneck. The proportion of type $a$ individuals in the last generation of the bottleneck is then
\begin{equation}\label{eq:drastic-final-freq	}
\tilde{X}^{\i,\mathbf{x}}(\tau):=\frac{1}{\size(\tau)} \sum_{j=1}^{\demes}\sum_{m=1}^{\infty} a^\i_{jm}(\tau)B_{jm}^{x_j},	
\end{equation}
where the $B_{jm}^{x_j}$'s are independent Bernoulli random variables with parameter $x_j$. Notice that, for every $j$, $a^\i_{jm}(\tau)=0$ whenever $l>\size(1)+c_i(\tau-1)$, hence the sums are actually finite.

We are now ready to write down the limit of the frequency process as $N \rightarrow \infty$.

\begin{theorem}\label{thm:drastic-diffusion}
Fix $N\in \N$, $\gamma>0$ and $L$ a probability measure on $\N$. Fix also a sequence $\{l_k^N\}_{k\in \N}$, a sequence $\{s_k^N\}_{k\in \N}$ and a function $\size:\N \rightarrow \N$ satisfying \ref{enumerate:1}, \ref{enumerate:2}, \ref{enumerate:3} respectively.  Let $A^\i(\tau,\cdot)$ be the distribution of the matrix $\mathbf{A}^\i(\tau)\in  \N_0^{\demes\times \N}$ of the final family sizes after a bottleneck affecting deme $i$ for $\tau$ generations with population size determined by $\size$. Let $\{\mathbf{X}^N(g)\}_{g\in \N_0}$ be the frequency process of the structured Wright--Fisher model with drastic bottlenecks parametrised by $N$, $\gamma$, $L$, and $\size$. Assume $\mathbf{X}^N(0)\xRightarrow[N\rightarrow \infty]{} \mathbf{X}(0)$ weakly. Then, for any $T>0$,
\begin{equation*}
\big(\mathbf{X}^N({\lfloor Nt\rfloor})\big)_{0 \leq t \leq T} \xRightarrow[N\rightarrow \infty]{d_{\lambda}} \big(\mathbf{X}(t)\big)_{0 \leq t \leq T}\,,
\end{equation*}
where $\xRightarrow{d_{\lambda}}$ denotes weak convergence in the topology induced by $d_{\lambda}$, and \\$\mathbf{X}(t)=(X_1(t),\ldots,X_{\demes}(t))$ is the unique strong solution of the following system of SDEs:
\begin{equation}\label{eq:drastic-diffusion}
\begin{split}
dX_i(t)=&\sum_{j\neq i}\frac{c_{ji}}{\omega_i}(X_j(t)-X_i(t))dt+\sqrt{\frac{1}{\omega_i}X_i(t)(1-X_i(t))}dB_i(t)\,\\&+\frac{1}{\demes}\int_{\mathcal{S}}\left(\frac{1}{\size(\tau)}\Big[\sum_{j=1}^{\demes}\sum_{m=1}^{\infty}a^\i_{jm}(\tau)\1_{\{u_{jm}\leq X_j(t^-)\}}\Big]-X_i(t^-)\right)\mathcal{N}_i(dt,d\tau,da,du)
\end{split}
\end{equation}
for $i=1,\ldots,\demes$. 
Here $(B_i(t))_{t \geq 0}$, $i=1,\ldots ,\demes$, are independent Brownian  motions,
\begin{equation*}
\mathcal{S}:=\N_0 \times \N_0^{\demes \times \N} \times [0,1]^{\demes \times \N},
\end{equation*}
 and $\mathcal{N}_i$ are independent Poisson random measures on $(0,\infty)\times \mathcal{S}$ with intensity given by $\gamma dt \otimes \mathfrak{A}^\i(d\tau,da)\otimes du$, where $du$ is the Lebesgue measure on $[0,1]^{\demes \times \N}$ and $\mathfrak{A}^\i$  is the semidirect product of $L$ and the probability kernel $A^\i(\cdot,\cdot)$.
\end{theorem}
\begin{remark}
The last term of \eqref{eq:drastic-diffusion} describes the change of frequency due to a drastic bottleneck; in particular the integrand reflects the expression \eqref{eq:drastic-final-freq	} and the uniform random variables are being used to generate the randomness from the Bernoulli’s. Additionally for any $K\subset \N_0$ and $A\subset \N_0^{\demes\times \N}$	 we have \[\mathfrak{A}^\i(K,A)=\int_K\int_A A^\i(\tau,da)L(d\tau).\]
\end{remark}

Before proving Theorem \ref{thm:drastic-diffusion} we have to verify that \eqref{eq:drastic-diffusion} has a unique strong solution; we prove this fact in Proposition \ref{prop:strong-sol}.
\begin{center}
\begin{figure}
\centering
\includegraphics[scale=0.75]{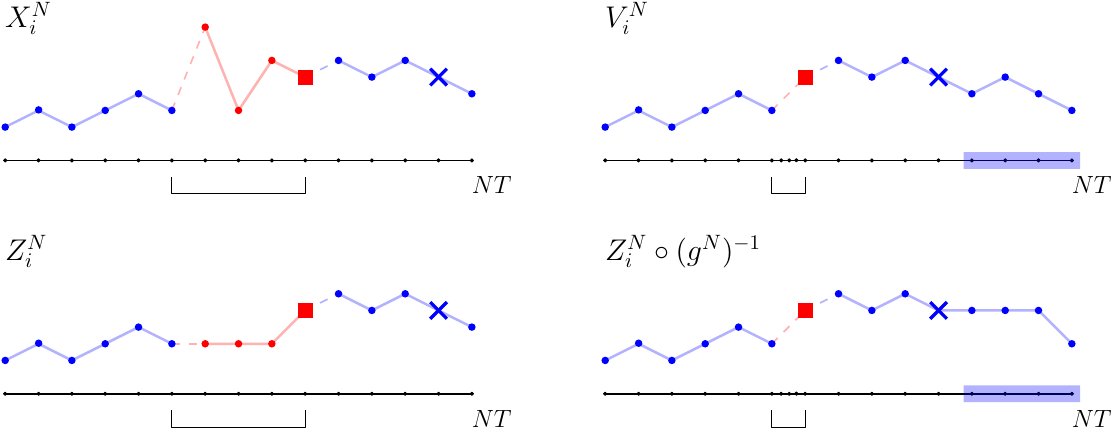}
\caption{In this figure we show a possible realisation of the discrete processes $X^N$, $V^N$, $Z^N$, \mbox{$Z^N\circ (g^N)^{-1}$}. In red, underlined by a square bracket, we draw what is happening during a bottleneck. In particular, we only represent the affected deme; with a red square we highlight the last generation of a bottleneck and with a blue cross the second to last generation $NT-1$ (left-hand side) which becomes $NT-\delta^N-1$ in the time scale that skips the interior of the bottlenecks (right-hand side). We highlight in a light blue rectangle the only area where $V^N$ and $Z^N \circ (g^N)^{-1}$ do not coincide.}\label{fig:X,Z,V processes}
\end{figure}
\end{center} 
\begin{proof}[Proof of Theorem \ref{thm:drastic-diffusion}]
For simplicity, we assume that $T$ is an integer, otherwise we can consider $\lfloor T \rfloor$. As we pointed out in Remark \ref{remark:J1-top}, the fluctuations of the pre-limiting frequency process $\mathbf{X}^N$ prevent us from obtaining convergence in the Skorokhod J1 topology. Therefore, we will work with an intermediate time-homogeneous Markov process $\mathbf{V}^N$ that we prove to be close, with respect to $d_{\lambda}$, to $\mathbf{X}^N$. Then we will prove that the generator of $\mathbf{V}^N$ converges to that of the limit $\mathbf{X}$. Note that the process $\mathbf{X}^N$ is not Markovian, so we could not use the convergence of generators for $\mathbf{X}^N$.

Let $k_0=0$ and $k_g$ be the $g$-th generation not in the set $\cup_{k=1}^{\infty}\{\underline{t}^N_k+1,\ldots ,\overline{t}^N_k-1\}$ and define the process $\{\mathbf{V}^N(g)\}_{g \in \N_0}$ as
\begin{equation*}
\mathbf{V}^N(g):=\mathbf{X}^N(k_g).
\end{equation*}
In other words, $\mathbf{V}^N$ sees each bottleneck as lasting only one generation (the last one), thus recovering the Markov property and the time homogeneity. 
Since $\mathbf{V}^N$ skips most of each bottleneck, when we run it up until the time $NT$, we see some additional generations that we do not see for $\mathbf{X}^N$, see Figure \ref{fig:X,Z,V processes},  and if a bottleneck occurs in this time frame, the two processes may be quite far apart. To prove that the distance between these two processes is small, we then need to verify that they see the same number of bottlenecks. This is the first step of the proof.
In the second step, we prove that 
$$d_{\lambda}\big((\mathbf{X}^N(\Nt))_{0\leq t \leq T},(\mathbf{V}^N(\Nt))_{0\leq t \leq T}\big)\xrightarrow{N\rightarrow \infty} 0 \quad \text{in probability}.$$
In the third step, we prove the convergence of the generator of $\mathbf{V}^N$ to that of the jump diffusion~\eqref{eq:drastic-diffusion}.\\

\noindent\textbf{Step 1.} Let $m^N_T$ be the index of the last bottleneck ending before generation $NT-1$, i.e.\
\begin{equation*}
m^N_T:=\sup\{k:\overline{t}^N_k < NT-1\},
\end{equation*}
and $\delta^N:=\sum_{k=1}^{m^N_T}(l^N_k-1)$ be the total number of these bottlenecks' generations that $\mathbf{V}^N$ skips. We want to prove that eventually there is a.s.\ no bottleneck in the time interval $[NT-1,NT+\delta^N]$; note that this implies that $\mathbf{X}^N$ and $\mathbf{V}^N$ see the same number of bottlenecks.
In order to do that, we show that the probability of the complementary event \[\mathfrak{B}_N:=\big\{\omega \in \Omega\, :\, \exists k\in \N \text{ s.t.\ } \{NT-1,NT\ldots ,NT+\delta^N\} \cap \{\underline{t}^N_k+1,\underline{t}^N_k+2,\ldots ,\overline{t}^N_k\}\neq \varnothing  \big\},\] tends to zero as $N \to \infty$. Heuristically, since bottlenecks happen every $O(N)$ generations, $m^N_T$ is a.s.\ finite. Hence the additional generations that we see in the path of $V^N$ is $o(N)$ which is too short a time for another bottleneck to occur.

We start by noticing the following:
\begin{itemize}
\item \label{condition A.2}  For all $ \varepsilon>0$ there exist $K_\varepsilon,N_\varepsilon\in \N$ such that, for all $ N\geq N_\varepsilon$,
\begin{equation*}
\P(l^N_1\geq K_\varepsilon)<\epsilon.	
 \end{equation*}
 This is true since, by Assumption \ref{enumerate:1}, $l^N_1$ converges in distribution to a probability measure $L$ on $\N$. Hence, if, by contradiction, there exists $\epsilon>0$ such that 
 \[\forall K,N\in \N \quad \exists N_K\geq N \text{ such that } \P(l^{N_K}_1\geq K)>\epsilon,\]
 we would have $L(\infty)\geq \epsilon$, which cannot be true.
\item \label{condition A.1} The same argument applies to $m^N_T$, i.e.\ for all $\varepsilon>0$ there exist $K_\varepsilon,N_\varepsilon\in \N$ such that for all $ N\geq N_\varepsilon$
\begin{equation*}
\P(m^N_T\geq K_\varepsilon)<\epsilon.	
 \end{equation*}
In fact, we have that $m^N_T\leq \max\{k:s^N_1+\ldots +s^N_k<NT+1\}$ a.s.\ and \ref{enumerate:2} implies that $\max\{k:s^N_1+\ldots +s^N_k<NT+1\}$ converges in distribution to a Poisson random variable with parameter $\gamma\,T$.
  \end{itemize}
Next, define the following event
\[\mathfrak{B}_{N,k}:=\big\{\omega \in \Omega\, : \{NT-1,NT,\ldots ,NT+\delta^N\} \cap \{\underline{t}^N_k+1,\underline{t}^N_k+2,\ldots ,\overline{t}^N_k\}\neq \varnothing  \big\},\]
and fix $\varepsilon>0$. We then know that there exist $K_\varepsilon,N_\varepsilon\in \N$ such that for all $ N\geq N_\varepsilon$
\begin{equation*}
\P(m^N_T\geq K_\varepsilon)<\frac{\epsilon}{4},	
 \end{equation*}
 and additionally there exist $\overline{K}_\varepsilon,\overline{N}_\varepsilon\in \N$ and $K^*_\varepsilon,N^*_\varepsilon\in \N$ such that for all $ N\geq \overline{N}_\varepsilon $ 
 \begin{equation*}
	\P(l^N_1\geq \overline{K}_\varepsilon)<\frac{\varepsilon}{4(K_\varepsilon)^2},
 \end{equation*}
 and for all $ N\geq N^*_\varepsilon$ 
 \begin{equation*}
 \P(l^N_1\geq K^*_\varepsilon)<\frac{\varepsilon}{4(K_\varepsilon)^2\overline{K}_\varepsilon}.
 \end{equation*}
We then have
\begin{align}
&\P\left(\mathfrak{B}_N\right)=\sum_{m\in \N}\sum_{d \in \N}\sum_{k=2}^{NT-1}\P\left(\mathfrak{B}_N\,,\,m^N_T=m\,,\, \delta^N=d\,,\,\overline{t}^N_{m^T_N}=NT-k\right)\notag\\
&=\sum_{m\in \N}\sum_{d \in \N}\sum_{k=2}^{NT-1}\P\left(\mathfrak{B}_{N,m+1},\,m^N_T=m\,,\, \delta^N=d\,,\,\overline{t}^N_{m}=NT-k \right) \notag\\
&\leq \sum _{m=1}^{K_\varepsilon}\sum_{d \in \N}\sum_{k=2}^{NT-1}\P\left(\mathfrak{B}_{N,m+1},\,m^N_T=m\,,\, \delta^N=d\,,\,\overline{t}^N_{m}=NT-k \right) + \P(m^N_T\geq K_\varepsilon)\notag \\
&\leq \sum _{m=1}^{K_\varepsilon}\sum_{d=1}^{K_\varepsilon \overline{K}_\varepsilon}\sum_{k=2}^{NT-1}\P\left(\mathfrak{B}_{N,m+1},\,m^N_T=m\,,\, \delta^N=d\,,\,\overline{t}^N_{m}=NT-k \right)\label{eq:NT-bott2-1}\\
&\quad+\sum_{m=1}^{K_\varepsilon}\sum_{d\geq K_\varepsilon\overline{K}_\varepsilon}\P(m^N_T=m,\delta^N=d)\label{eq:NT-bott2-2}\\
&\quad + \P(m^N_T\geq K_\varepsilon)\label{eq:NT-bott2-3}.
\end{align}
Now for \eqref{eq:NT-bott2-3} we have, for all $ N\geq N_\varepsilon$,

\begin{equation}\label{eq:mNT}
\begin{split}
\P(m^N_T\geq K_\varepsilon)< \frac{\varepsilon}{4},
\end{split}
\end{equation}
for \eqref{eq:NT-bott2-2} we get, for all $ N\geq \overline{N}_\varepsilon$,
\begin{equation*}
\begin{split}
&\sum_{m=1}^{K_\varepsilon}\sum_{d\geq K_\varepsilon\,\overline{K}_\varepsilon}\P(m^N_T=m,\delta^N=d)=\sum_{m=1}^{K_\varepsilon}\P\left(\sum_{k=1}^{m_T^N}(l^N_k-1)\geq K_\varepsilon\,\overline{K}_\varepsilon,\, m^N_T=m  \right)
\\&\leq K_\varepsilon \,\P\left(\sum_{k=1}^{K_\varepsilon}l^N_k\geq K_\varepsilon\,\overline{K}_\varepsilon\right)\leq K_\varepsilon\,\P\left(\exists k\in \{1,\ldots ,K_\varepsilon\}:l^N_k\geq \overline{K}_\varepsilon \right)
\\&\leq (K_\varepsilon)^2 \P(l^N_1\geq \overline{K} _\varepsilon)<\frac{\varepsilon}{4}.
\end{split}
\end{equation*}
Note that the last two estimates imply also that $\delta^N/N\to 0$ in probability.
Finally, we focus on~\eqref{eq:NT-bott2-1} and get, for all $N\geq N^*_\varepsilon$,
\begin{equation*}
\begin{split}
&\sum _{m=1}^{K_\varepsilon}\sum_{d=1}^{K_\varepsilon \overline{K}_\varepsilon}\sum_{k=2}^{NT-1}\P\left(\mathfrak{B}_{N,m+1},\,m^N_T=m\,,\, \delta^N=d\,,\,\overline{t}^N_{m}=NT-k \right)\\
&=\sum _{m=1}^{K_\varepsilon}\sum_{d=1}^{K_\varepsilon \overline{K}_\varepsilon}\sum_{k=2}^{NT-1}\sum_{l=1}^{\infty}\P\left(k-l-1 \leq s^N_{m+1}<k+d\,,\, l^N_{m+1}=l\,,\, \delta^N=d\,,\,\overline{t}^N_{m}=NT-k \right)\\
&\leq \sum _{m=1}^{K_\varepsilon}\sum_{d=1}^{K_\varepsilon \overline{K}_\varepsilon}\sum_{k=2}^{NT-1}\sum_{l=1}^{K^*_\varepsilon}\P\left(k-l-1 \leq s^N_{m+1}<k+d\right) \P(l^N_{m+1}=l)\P\left(\overline{t}^N_{m}=NT-k \right)\\
&\quad +\sum _{m=1}^{K_\varepsilon}\sum_{d=1}^{K_\varepsilon \overline{K}_\varepsilon}\P(l^N_{m+1}\geq K^*_\varepsilon).
\end{split}
\end{equation*}
Since the variables $s^N_k$ are geometrically-distributed with parameter $\gamma/N$, the above expression is equal to
\begin{equation*}
\begin{split}
&\sum _{m=1}^{K_\varepsilon}\sum_{d=1}^{K_\varepsilon \overline{K}_\varepsilon}\sum_{k=2}^{NT-1}\sum_{l=1}^{K^*_\varepsilon}\sum_{s=k-l-1}^{k+d-1}\frac{\gamma}{N}\left(1-\frac{\gamma}{N}\right)^{s-1}\P(l^N_{m+1}=l)\P\left(\overline{t}^N_{m}=NT-k \right)\\
&\quad +\sum _{m=1}^{K_\varepsilon}\sum_{d=1}^{K_\varepsilon \overline{K}_\varepsilon}\P(l^N_{m+1}\geq K^*_\varepsilon)\\
&\leq \sum _{m=1}^{K_\varepsilon}\sum_{d=1}^{K_\varepsilon \overline{K}_\varepsilon}\frac{\gamma}{N}K^*_\varepsilon\left(K^*_\varepsilon+d\right)+\sum _{m=1}^{K_\varepsilon}\sum_{d=1}^{K_\varepsilon \overline{K}_\varepsilon}\P(l^N_{m+1}\geq K^*_\varepsilon)\\
&\leq \frac{\gamma}{N}(K_\varepsilon)^2\overline{K}_\varepsilon K^*_\varepsilon\left(K^*_\varepsilon+K_\varepsilon \overline{K}_\varepsilon\right)+(K_\varepsilon)^2\overline{K}_\varepsilon \P(l^N_{m+1}\geq K^*_\varepsilon)<\frac{\varepsilon}{2},
\end{split}
\end{equation*}
where in the last step we assumed, without loss of generality, that for all $N\geq N^*_\varepsilon$ one has $$\frac{\gamma}{N}(K_\varepsilon)^2\overline{K}_\varepsilon K^*_\varepsilon\left(K^*_\varepsilon+K_\varepsilon \overline{K}_\varepsilon\right)<\frac{\varepsilon}{4}.$$
Putting everything together we have $\P(\mathfrak{B}_N)\longrightarrow \, 0  \text{ as }\,N\to \infty.$ \\

\noindent \textbf{Step 2.}
Our goal is to prove that the distance $d_\lambda$ between $\{\mathbf{X}^N(g)\}_{g \in \N_0}$ and $\{\mathbf{V}^N(g)\}_{g \in \N_0}$ is small in probability as $N \rightarrow \infty$. To do so, we define an intermediate process $\mathbf{Z}^N$ which, at each bottleneck, remains constant until the second to last generation and takes the same value as $\mathbf{X}^N$ in the last generation. More precisely, let $\pi^N$ be the random projection $\pi^N(0)=0$ and $\pi^N(g)=\max\{k \in \N_0: k \leq g, k \notin \cup_{m=1}^{\infty}\{\underline{t}^N_m+1,\ldots ,\overline{t}^N_m-1\}\}$; we define the process 
\begin{equation*}
\mathbf{Z}^N(g):=\mathbf{X}^N(\pi^N(g)),
\end{equation*}
see Figure \ref{fig:X,Z,V processes} for a visual representation of this process. 
By the triangle inequality
\begin{equation*}
\begin{split}
&d_{\lambda}\big((\mathbf{X}^N(\Nt))_{0\leq t \leq T},(\mathbf{V}^N(\Nt))_{0\leq t \leq T}\big)\\&\leq d_{\lambda}\big((\mathbf{X}^N(\Nt))_{0\leq t \leq T},(\mathbf{Z}^N(\Nt))_{0\leq t \leq T}\big)+d_{\lambda}\big((\mathbf{Z}^N(\Nt))_{0\leq t \leq T},(\mathbf{V}^N(\Nt))_{0\leq t \leq T}\big).
\end{split}
\end{equation*}
The distance $d_\lambda$ between $\mathbf{X}^N$ and $\mathbf{Z}^N$ is small because they coincide other that during bottlenecks, which is a negligible set of times in the limit $N \to \infty$; on the other hand we will define a proper time change $f^N$ that shows that also $\mathbf{Z}^N$ and $\mathbf{V}^N$ are close with respect to $d_\lambda$.

We start by computing $d_{\lambda}\big((\mathbf{V}^N(\Nt))_{0\leq t \leq T},(\mathbf{Z}^N(\Nt))_{0\leq t \leq T}\big)$. To this end, we consider the function $g^N:[0,NT]\to[0,NT]$ that is the piecewise linear function having slope
\begin{equation*}\left\{
\begin{array}{lll}
1/l^N_k  & \,\,\, \text{if} \quad\lfloor t\rfloor \in [\underline{t}_k^N,\overline{t}^N_k),\\
1+\sum_{k=1}^{m^N_T}(l^N_k-1) &\,\,\, \text{if} \quad\lfloor t\rfloor \in [NT-1,NT],\\
1 &\,\,\,  \text{elsewhere},
 \end{array}\right.
\end{equation*} 
where we recall that $m^N_T=\sup\{k:\overline{t}^N_k < NT-1\}$, see Figure \ref{fig:function}. Hence, $(g^N)^{-1}$ accelerates time at the occurrence of bottlenecks so that their length is reduced to one unit. Finally, when we reach the last generation before the time horizon, we slightly decrease the slope to get $(g^N)^{-1}(NT)=NT$. 
\begin{figure}
\begin{center}
\includegraphics[scale=0.5]{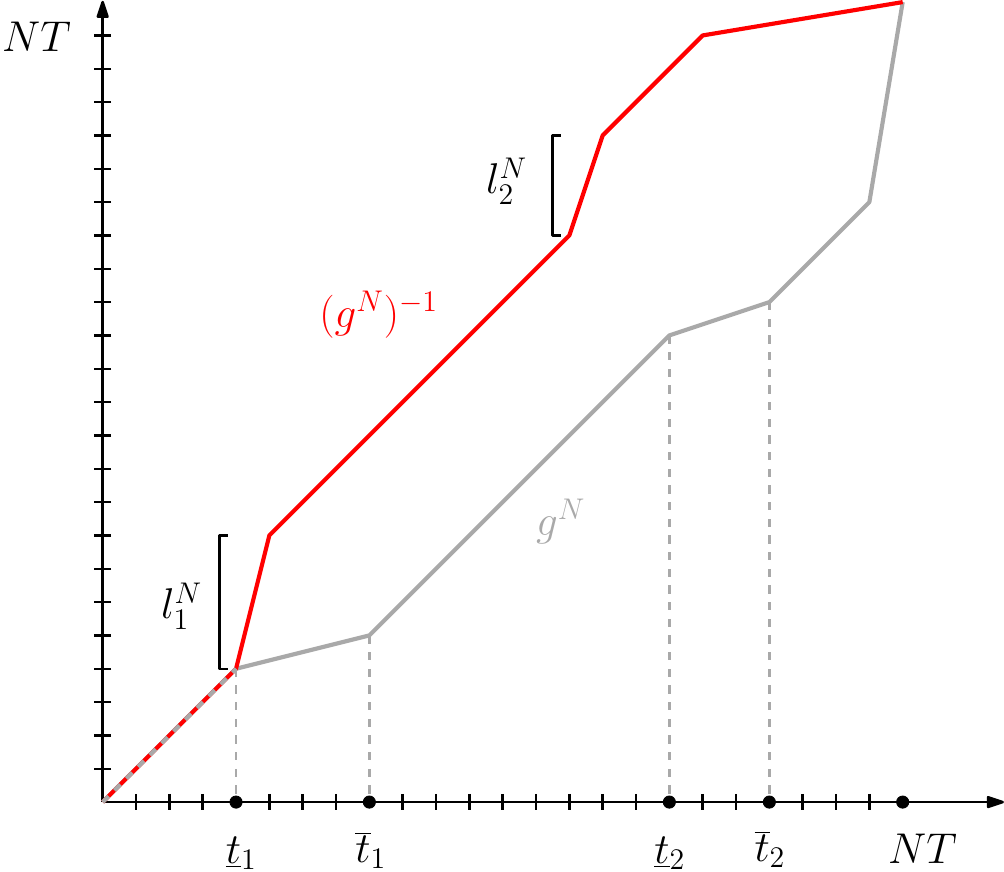}
\end{center}
\caption{A realization of the functions $g^N$ and $(g^N)^{-1}$ when $m^N_T=2$.}\label{fig:function}	
\end{figure}

Given this function, we define the time change $f^N\in \mathcal{F}$ as 
\begin{equation}\label{eq:time-change}
f^N(t)=g^N(Nt)/N \quad \text{ for } t\in[0,T].
\end{equation}
Note that $\lfloor N(f^N)^{-1}(\cdot)\rfloor=\lfloor (g^N)^{-1}(N\cdot)\rfloor$. To bound $\overline{d}_{\lambda}(\mathbf{V}^N(\Nt),\mathbf{Z}^N(\Nt))$, see Definition~\ref{def:dmu}, we choose $A=[0,T)$ and $f=(f^N)^{-1}$ and for $\overline{d}_{\lambda}(\mathbf{Z}^N(\Nt),\mathbf{V}^N(\Nt))$, we choose $A=[0,T)$ and $f=f^N$; thus
\begin{equation}\label{eq:dist-V-Z}
\begin{split}
&d_{\lambda}\big((\mathbf{V}^N(\Nt))_{0\leq t \leq T},(\mathbf{Z}^N(\Nt))_{0\leq t \leq T}\big)\leq \norm{\mathbf{V}^N(\lfloor N\cdot\rfloor)-\mathbf{Z}^N \big( \lfloor N(f^N)^{-1}(\cdot)\rfloor \big)}_{\infty}\\&+\norm{\text{Id}-(f^N)^{-1}}_{\infty}+\norm{\mathbf{V}^N(NT)-\mathbf{Z}^N\big(N(f^N)^{-1}(T)\big)}.	
\end{split}
\end{equation}
We have that $\mathbf{V}^N$ and $\mathbf{Z}^N\circ(f^N)^{-1}$ coincide up until the blue cross of the right-hand side of Figure~\ref{fig:X,Z,V processes}, i.e.\ 
\begin{equation*}
\mathbf{V}^N(\Nt)=\mathbf{Z}^N(\lfloor N(f^N)^{-1}(t)\rfloor) \quad \forall t< T-\delta^N/N.
\end{equation*}
Therefore, to compute the uniform norm in \eqref{eq:dist-V-Z}, we can focus just on the interval $[T-\delta^N/N,T)$ in which $\mathbf{Z}^N$ is constant and equal to $\mathbf{V}^N(N(T-\delta^N/N))$. Conditioning on $\mathfrak{B}_N^c$, we then have
\begin{equation*}
\begin{split}
&\P\left(\norm{\mathbf{V}^N(\lfloor N\cdot\rfloor)-\mathbf{Z}^N\big(\lfloor N(f^N)^{-1}(\cdot)\rfloor \big)}_{\infty}>\epsilon\right)\\&\leq \P\left(\sup_{t\in [0,1)}\left\{\mathbf{V}^N(\lfloor N(T-t\delta^N/N)\rfloor)-\mathbf{V}^N(N(T-\delta^N/N))\right\}>\epsilon \right),
\end{split}
\end{equation*}
which tends to zero by the Portemanteau theorem, \cite[Theorem 13.16]{Klenke}. Indeed, $\delta^N/N\to 0$ in probability and we will prove in the following step that $\mathbf{V}^N(\lfloor N\cdot \rfloor)$ converges to a diffusion $\mathbf{X}$ whose only discontinuities are due to bottlenecks; hence we can apply the above mentioned theorem using continuous test functions such that $\sup_{t\in [0,1)}\left\abs f(T-t\delta)-f(T-\delta)\right\abs >\epsilon$ for some small $\delta>0$.

Hence $\norm{\mathbf{V}^N(\lfloor N\cdot\rfloor)-\mathbf{Z}^N \big( \lfloor N(f^N)^{-1}(\cdot)\rfloor \big)}_{\infty} \to 0$ in probability.
The same observation also shows that $\norm{\mathbf{V}^N(NT)-\mathbf{Z}^N\big(N(f^N)^{-1}(T)\big)}$ converges to zero in probability. Finally, we have
\begin{equation*}
\norm{\text{Id}-(f^N)^{-1}}_{\infty}=\frac{\delta^N}{N}\xrightarrow{N\rightarrow \infty}0 \quad  \text{in probability},
\end{equation*}
so that we can conclude 
\begin{equation*}
d_{\lambda}\big((\mathbf{V}^N(\Nt))_{0\leq t \leq T},(\mathbf{Z}^N(\Nt))_{0\leq t \leq T}\big)\xrightarrow{N\rightarrow \infty}0 \quad  \text{in probability}.
\end{equation*}
To bound $d_{\lambda}\big((\mathbf{X}^N(\Nt))_{0\leq t \leq T},(\mathbf{Z}^N(\Nt))_{0\leq t \leq T}\big)$, we consider the set \[U^N=\bigcup_{k=1}^{m^N_T}\Big(\,\frac{\underline{t}^N_k}{N},\frac{\overline{t}^N_k}{N}\Big].\] 
We have
\begin{equation*}
\norm{\1_{(U^N)^c}\big(\mathbf{Z}^N(N\,\cdot)-\mathbf{X}^N(N\, \cdot)\big)}_{\infty}= 0 \quad \text{in probability,}
\end{equation*}
since in $(U^N)^c$ there are no bottlenecks so that $\mathbf{Z}^N$ and $\mathbf{X}^N$ coincide. As 
\begin{equation*}
\lambda(U^N)\xrightarrow{N\rightarrow \infty} 0 \quad \text{in probability},
\end{equation*}
we get
\begin{equation*}
d_{\lambda}\big((\mathbf{X}^N(\Nt))_{0\leq t \leq T},(\mathbf{Z}^N(\Nt))_{0\leq t \leq T}\big)\xrightarrow{N\rightarrow \infty} 0 \quad \text{in probability,}	
\end{equation*}
concluding the second step of the proof.
\vskip 11pt
\noindent \textbf{Step 3.}
The idea is to prove the convergence of the generator of $(\mathbf{V}^N(\lfloor Nt\rfloor))_{t \geq 0}$ to that of $(\mathbf{X}(t))_{t \geq 0}$.
Since in Proposition \ref{prop:Feller} we prove that the limiting process is Feller continuous, weak convergence of $(\mathbf{V}^N(\lfloor Nt\rfloor))_{0 \leq t \leq T}$ to $(\mathbf{X}(t))_{0 \leq t \leq T}$ follows by \cite[Theorem 19.25 and 19.28]{Kallenberg}. The generator of \eqref{eq:drastic-diffusion}, for any $f \in C^3([0,1]^\demes)$, can be written as

\begin{equation}\label{eq:drastic-frequency-generator}
\begin{split}
\mathcal{A}f(\mathbf{x})={}&  \sum_{i=1}^\demes \left(\sum_{j\neq i}\frac{c_{ji}}{\omega_i}(x_j-x_i)\frac{\partial f}{\partial x_i}(\mathbf{x})+\frac{1}{2\omega_i}x_i(1-x_i)\frac{\partial f^2}{\partial x_i^2}(\mathbf{x})\right)\\
&+\sum_{i=1}^\demes \frac{\gamma}{\demes}\sum_{\tau\geq 1}L(\tau)\sum_{k=0}^{\size(\tau)}\P\left(\tilde{X}^{\i,x}(\tau)=\frac{k}{\size(\tau)}\right)\\&\qquad\times \left(f\left(x_1,\ldots,x_{i-1},\frac{k}{\size(\tau)},x_{i+1},\ldots,x_\demes\right)-f(x_1,\ldots,x_\demes)\right),
\end{split}
\end{equation}
where 
\begin{equation}\label{eq:drastic-X-tilda}
\tilde{X}^{\i,\mathbf{x}}(\tau):=\frac{1}{\size(\tau)} \sum_{j=1}^{\demes}\sum_{m=1}^{\infty} a^\i_{jm}(\tau)B_m^{x_j},	
\end{equation}
and the  $B_m^{x_j}$'s are independent Bernoulli random variables with parameter $x_j$.
The generator of $\mathbf{V}^N$ is, for any $f \in C^3([0,1]^\demes)$,
\begin{equation*}
\begin{split}
\mathcal{A}^Nf(\mathbf{x}):={}&N\E_{\mathbf{x}}\Big[f(\mathbf{V}^N(1))-f(\mathbf{x})\Big]=N\left(1-\frac{\gamma}{N}\right)\mathcal{L}^N(\mathbf{x})+N\frac{\gamma}{N}\mathcal{K}^N(\mathbf{x}),
\end{split}
\end{equation*}
where $\E_{\mathbf{x}}[\cdot]=\E[\cdot|\mathbf{V}^{N}(0)=\mathbf{x}]$, and

\begin{equation*}
\begin{split}
\mathcal{L}^Nf(\mathbf{x}):=\E_{\mathbf{x}}\Big[f(\mathbf{V}^N_1)-f(\mathbf{x})|s_1^N>1\Big]
\end{split}
\end{equation*}
is the Wright--Fisher part of the generator, i.e.\ if there is no bottleneck in the first time step, and
\begin{equation*}
\begin{split}
\mathcal{K}^Nf(\mathbf{x})&:=\E_{\mathbf{x}}\Big[f(\mathbf{V}^N_1)-f(\mathbf{x})|s_1^N=1\Big]\\&=\sum_{i=1}^\demes \frac{\gamma}{\demes}\sum_{\tau \geq 1}\P(l_1^N=\tau)\sum_{k=0}^{\size(\tau)}\P\left(\tilde{X}^{\i,x}(\tau)=\frac{k}{\size(\tau)}\right)\\&\qquad \times\left(f\left(x_1,\ldots,x_{i-1},\frac{k}{\size(\tau)},x_{i+1},\ldots,x_\demes\right)-f(x_1,\ldots,x_\demes)\right).
\end{split}
\end{equation*}
is the jump part, i.e.\ if a bottleneck occurs.

By \cite[Theorem 3.6]{Eth12}, we get that $N\mathcal{L}^Nf \xrightarrow{N\rightarrow \infty} \mathcal{L}f$ uniformly for every $f \in C^3([0,1]^\demes)$, where
\begin{equation*}
\mathcal{L}f(\mathbf{x}):= \sum_{i=1}^\demes \left(\sum_{j\neq i}\frac{c_{ji}}{\omega_i}(x_j-x_i)\frac{\partial f}{\partial x_i}(\mathbf{x})+\frac{1}{2\omega_i}x_i(1-x_i)\frac{\partial f^2}{\partial x_i^2}(\mathbf{x})\right).
\end{equation*}
The uniform convergence of the jump part holds since the distribution of $l_1^N$ converges to $L$.
\end{proof}

\subsection{Duality}
We now focus on the backward in time dynamics and introduce the block-counting process of a multi-type $\Xi$-coalescent presenting simultaneous multiple mergers and migrations. We then prove that this process is the moment dual of the frequency process \eqref{eq:drastic-diffusion}. Throughout this section we will work with the limiting model where the population size $N$ is equal to $\infty$.

We want to trace back the ancestral lines of a sample of $\mathbf{n}=(n_1,\ldots,n_\demes)$ individuals, where $n_i$ is the number of individuals sampled from deme $i$.

We begin by describing a drastic bottleneck backward in time.
We denote by $\xi^\i(\tau):=\size(1)+c_i\cdot (\tau-1)$ the total number of families arising during a drastic bottleneck of length $\tau$ affecting deme $i$, and by $\mathbf{p}^\i(\tau)=(p^\i_{jm}(\tau))_{1\leq j \leq \demes,\, 1\leq m \leq \xi^\i(\tau)}$ their respective weights as defined in \eqref{eq:relative-size}.	 
Looking backward in time, if we have $n_i$ lines ``before'' the event, these choose which family $jm$ to go into with probability $p^\i_{jm}(\tau)$. After the bottleneck, families $jm$ for $m=1,\ldots,\xi^\i(\tau)$ migrate to deme $j$. In other words, we allocate $n_i$ balls to $\demes \times \xi^\i(\tau)$ labelled and weighted boxes with weights $\mathbf{p}^\i(\tau)$. We define the following random variables
\begin{equation*}
\tilde{N}_j^{\i,\mathbf{n}}(\tau)= \sum_{m=1}^{\xi^\i(\tau)} \1(\{\text{box $jm$ is non-empty}\}).
\end{equation*}
Consider the continuous-time Markov process $(\mathbf{N}(t))_{t \geq 0}=(N_1(t),\ldots,N_\demes(t))_{t \geq 0}$ on $\N^D$, characterised by the following transitions from a state $\mathbf{n}=(n_1,\ldots ,n_\demes)\in\N^D$:
\begin{itemize}
\item $\mathbf{n}\mapsto \mathbf{n}-\mathbf{e}_i$ with rate
\begin{equation*}
	\frac{1}{\omega_i}\binom{n_i}{2}+\frac{\gamma}{\demes}\sum\limits_{\tau\geq 1}L(\tau)\P\big(\tilde{N}_i^{\i,n}(\tau)=n_i-1,\,\tilde{N}_j^{\i,n}(\tau)=0 \,\,\forall j\neq i\big),
\end{equation*}
\item $\mathbf{n}\mapsto \mathbf{n}-\mathbf{e}_i+\mathbf{e}_j$ with rate
\begin{equation*}
	\frac{c_{ji}}{\omega_i}n_i+\frac{\gamma}{\demes}\sum\limits_{\tau\geq 1}L(\tau)\P\big(\tilde{N}_i^{\i,n}(\tau)=n_i-1, \tilde{N}_j^{\i,n}(\tau)=1\big),
\end{equation*}
\item $\mathbf{n}\mapsto \mathbf{n}-(n_i-a)\mathbf{e}_i+\sum_{j\neq i}b_j\mathbf{e}_j$ with rate
\begin{equation*}
	\frac{\gamma}{\demes}\sum\limits_{\tau\geq 1}L(\tau)\P\big(\tilde{N}_i^{\i,n}(\tau)=a,\,\tilde{N}_j^{\i,n}(\tau)=b_j \,\,\forall j\neq i\big),
\end{equation*}
\end{itemize}
where $\mathbf{e}_j$ is the $j$-th canonical unit vector.

\begin{remark}
This is the candidate process counting the number of ancestral lineages of a sample taken from the present generation. The first two transitions represent binary mergers (in one deme) and single migrations, respectively. These are all the possible transitions of a structured Kingman coalescent, \cite{Hud90,Herbots_1998}. The occurrence of bottlenecks, however, gives rise to one other possible transition: simultaneous multiple mergers with simultaneous multiple migrations.
Indeed, since by construction simultaneous multiple mergers (see Figure \ref{fig:drastic-frequency}) have positive probability, we can see $(\mathbf{N}(t))_{t\geq 0}$ as the block-counting process not of a $\Lambda$-coalescent but of a multi-type $\Xi$-coalescent.
Due to the fact that bottlenecks affect one deme at a time, these multiple mergers happen in the affected deme only, and the simultaneous migrations are only from the affected deme to the others.
\end{remark}
In the following theorem, we prove that $(\mathbf{N}_t)_{t \geq 0}$ satisfies a moment duality relation with the frequency process $(\mathbf{X}_t)_{t \geq 0}$.

\begin{theorem}\label{thm:drastic-frequency}
For every $\mathbf{x}\in[0,1]^\demes$, $\mathbf{n}\in \N_0^\demes$, we have
\begin{equation*}
\E_{\mathbf{x}}\Bigg[\prod_{i=1}^\demes X_i(t)^{n_i}\Bigg]=\E_{\mathbf{n}}\Bigg[\prod_{i=1}^\demes x_i^{N_i(t)}\Bigg].	
\end{equation*}
\end{theorem}

\begin{proof}
For any function $f:\N_0^\demes \rightarrow \R$, the infinitesimal generator of the block-counting process $\mathbf{N}$ is given by
\begin{equation*}
\begin{split}\label{eq:drastic-coalescent-generator}
\mathcal{G}f(\mathbf{n}) := \sum_{i=1}^\demes\Bigg[&\left(\frac{1}{\omega_i}\binom{n_i}{2}(f(\mathbf{n}-\mathbf{e}_i)-f(\mathbf{n}))+\frac{c_{ji}}{\omega_i}n_i(f(\mathbf{n}-\mathbf{e}_i+\mathbf{e}_j)-f(\mathbf{n}))\right)\\
&+\frac{\gamma}{\demes}\sum_{\tau\geq 1}L(\tau)\sum_{\substack{a,(b_j)\,:\\ a+\sum_{j\neq i} b_j\leq n_i}}\P(\tilde{N}_i^{\i,n}(\tau)=a,\,\tilde{N}_j^{\i,n}(\tau)=b_j \,\,\forall j\neq i)\\
&\phantom{+\frac{\gamma}{\demes}\sum_{\tau\geq 1}L(\tau)\sum_{a+\sum_{j\neq i} b_j\leq n_i}}\times \Big(f\big(\mathbf{n}-n_i\mathbf{e}_i+a\mathbf{e}_i+\sum_{j\neq i}b_j\mathbf{e}_j\big)-f(\mathbf{n})\Big)\Bigg].
\end{split}
\end{equation*}
Let $\mathcal{A}$ be the generator of the allele-frequency process $(\mathbf{X}_t)_{t\geq 0}$, as defined in \eqref{eq:drastic-frequency-generator}.
Our aim is to prove that, for $h:[0,1]^\demes \times\N^\demes\rightarrow\R$, defined as $h(\mathbf{x},\mathbf{n})=\prod_{i=1}^\demes x_i^{n_i}$,
\begin{equation*}
\mathcal{A}\,h(\mathbf{x},\mathbf{n})=\mathcal{G}\,h(\mathbf{x},\mathbf{n}),
\end{equation*}
when $\mathcal{A}$ acts on the first argument and $\mathcal{G}$ on the second. Since $(\mathbf{X}(t))_{t\geq 0}$ is Feller continuous (see Proposition \ref{prop:Feller}), duality then follows from \cite[Proposition 1.2]{JK14}.

If we consider the \emph{Wright--Fisher component} of $\mathcal{A}$,
\begin{equation*}
\mathcal{A^{\text{WF}}}f(\mathbf{x})=\sum_{i=1}^\demes \left(\frac{1}{2\omega_i}x_i(1-x_i)\frac{\partial f^2}{\partial x_i^2}(\mathbf{x})+\sum_{j\neq i}\frac{c_{ji}}{\omega_i}(x_j-x_i)\frac{\partial f}{\partial x_i}(\mathbf{x})\right),\end{equation*}
and the \emph{Kingman component} of $\mathcal{G}$
\begin{equation*}
\begin{split}
\mathcal{G}^{\text{K}}f(\mathbf{n})=\sum_{i=1}^\demes\left[\frac{1}{\omega_i}\binom{n_i}{2}(f(\mathbf{n}-\mathbf{e}_i)-f(\mathbf{n}))+\frac{c_{ji}}{\omega_i}n_i(f(\mathbf{n}-\mathbf{e}_i+\mathbf{e}_j)-f(\mathbf{n}))\right],
\end{split}
\end{equation*}
we have that 
\begin{equation*}
\begin{split}
&\mathcal{A^{\text{WF}}}\,h(\mathbf{x},\mathbf{n})\\
&= \sum_{i=1}^\demes \left(\frac{1}{2\omega_i}x_i(1-x_i)n_i(n_i-1)x_i^{n_i-2}\prod_{j\neq i }x_j^{n_j}+\sum_{j \neq i}\frac{c_{ji}}{\omega_i}(x_j-x_i)n_ix^{n_i-1}\prod_{k\neq i}x_k^{n_k}\right)\\
&=\sum_{i=1}^\demes \left(\frac{1}{\omega_i}\binom{n_i}{2} \Big(x_i^{n_i-1}\prod_{j\neq i }x_j^{n_j}-\prod_{j=1}^\demes x_j^{n_j} \Big) + \sum_{j \neq i}\frac{c_{ji}}{\omega_i}n_i\Big(x_i^{n_i-1}x_j^{n_j+1}\prod_{k\neq i,j }x_k^{n_k}-\prod_{j=1}^\demes x_j^{n_j} \Big)\right)\\
&=\mathcal{G}^{\text{K}}\,h(\mathbf{x},\mathbf{n}).
\end{split}
\end{equation*}
Hence, we can now focus on the jump parts. In particular, we want to prove that for any $i \in {1,\ldots,\demes}$ and $\tau \in \N_0$, we have
\begin{equation}\label{eq:generator-duality}
\begin{split}
&\sum_{k=0}^{\size(\tau)}\P\left(\tilde{X}^{\i,x}(\tau)=\frac{k}{\size(\tau)}\right)h\left(\mathbf{x}+\Big(\frac{k}{\size(\tau)}-x_i\Big)\mathbf{e}_i,\mathbf{n}\right)\\&=	\sum_{\substack{a,(b_j)\,:\\ a+\sum_{j\neq i} b_j\leq n_i}}\P(\tilde{N}_i^{\i,n}(\tau)=a,\,\tilde{N}_j^{\i,n}(\tau)=b_j \,\,\forall j\neq i)\\&\quad \quad \quad \times h\big(\mathbf{x},\mathbf{n}-n_i\mathbf{e}_i+a\mathbf{e}_i+\sum_{j\neq i}b_j\mathbf{e}_j\big),
\end{split}
\end{equation}
for all $\mathbf{x}\in [0,1]^\demes$ and $\mathbf{n}\in \N_0^\demes$, where $\tilde{X}^{\i,\mathbf{x}}(\tau)$ is defined in \eqref{eq:drastic-X-tilda}. To do that, we use the method of \emph{sampling duality}, \cite{Mohle99}. In particular, we show that both sides express the probability of sampling only individuals of type $a$ after (forward in time) a bottleneck in the pre-limiting model. Assume that a bottleneck of length $\tau$ takes place, and that there are $\mathbf{n}=(n_1,\ldots,n_i,\ldots,n_\demes)$ lines right before (backward in time) the bottleneck. Now we compute the probability that all these individuals are of type $a$. Firstly, observe that the sample is made of just type-$a$ individuals if and only if they all chose a type-$a$ parent in the first (backward in time) bottleneck generation. As the frequency of type $a$ in the previous generation is given by $(x_1,\ldots ,x_{i-1},\tilde{X}^{\i,\mathbf{x}}(\tau),x_{i+1},\ldots ,x_\demes)$ and the choices of parents are independent, this happens with probability $\big(\tilde{X}^{\i,\mathbf{x}}(\tau)\big)^{n_i}\prod_{j\neq i} x_j^{n_j} $; its expected value is the left-hand side of \eqref{eq:generator-duality}. 
A second way to compute the probability of this event is to go back until the last step of the bottleneck, in which our sample has $\mathbf{n}-n_i\mathbf{e}_i+\tilde{N}_i^{\i,\mathbf{n}}(\tau)\mathbf{e}_i+\sum_{j\neq i}\tilde{N}_j^{\i,\mathbf{n}}(\tau)\mathbf{e}_j$ ancestors. Therefore, the initial individuals are all of type $a$ if and only if all these ancestors have type-$a$ parents. This event has probability $x_i^{\tilde{N}^{\i,\mathbf{n}}_i(\tau)}\prod_{j\neq i}x_j^{n_j+\tilde{N}^{\i,\mathbf{n}}_j(\tau)}$; its expected value is the right-hand side of \eqref{eq:generator-duality}, and this concludes the proof.
\end{proof}

\section{Soft bottlenecks}\label{section:soft} 
In this section we define a model presenting \emph{soft} bottlenecks, i.e.\ during which the population size grows to infinity but is $o(N)$. As in the previous section, we want the bottlenecks, once time is rescaled by a factor $N$, to have negligible duration so that their effect will be that of an instantaneous jump. Hence we assume that the bottlenecks' lengths grow at the same speed as their sizes. This is the right scaling to consider because, if the duration were shorter, bottlenecks would have no effect on the genealogy that would simply look like a time-rescaled structured Kingman coalescent. If, on the other hand, we were to consider durations of higher order, the effect would be too strong, and, for any given population sample, after a bottleneck, there would be a.s.\ no genealogical line left in the affected island. 
The genealogy that we instead observe (Theorem \ref{thm:soft-convergence-coalescent}) is a coalescent with simultaneous multiple mergers and migrations.

 We consider bottlenecks determined by parameters satisfying the following assumptions: 
 \begin{enumerate}[label=\textbf{S.\arabic*}]
\item \label{senumerate:1} A deterministic sequence $\{l_k^N\}_{k\in \N}$ with $l_k^N\rightarrow \infty$ but $l_k^N/N \rightarrow 0$ as $N\rightarrow \infty $ for every $k\in \N$.
\item \label{senumerate:2} A sequence $\{s_k^N\}_{k\in \N}$ of i.i.d.\ random variables following a geometric distribution with parameter $\gamma/N$ for some $\gamma>0$.
\item \label{senumerate:3} A sequence of functions $\presize_k:[0,l^N]\rightarrow \N$ such that, for some continuous function \mbox{$\varphi:[0,1]\rightarrow (0,\infty)$}, for all $k\in \N$,
\begin{equation}\label{eq:presize_convergence}
\lim_{N \to \infty}\frac{\presize_k(\lfloor l_k^Nt\rfloor)}{l_k^N}=\varphi(t) \hspace{2em} \forall t\in[0,1].
\end{equation}
We define the \emph{effective length} of the bottlenecks as
\begin{equation}\label{eq:lambda}
\efflen:=\int_0^{1}\frac{1}{\varphi(t)}dt>0.
\end{equation}
\end{enumerate}

\begin{remark}
Using Conditions \eqref{eq:presize_convergence} and \eqref{eq:lambda}, we will prove that the effect of a soft bottleneck on the genealogy of a sample is that of a Kingman coalescent with emigration running in the affected deme for $\efflen$ units of time, while the other demes remain unaltered. Since the bottlenecks' lengths are negligible relative to the diffusion time scale $N$, we observe all the resulting events simultaneously.
\end{remark}
\subsection{Coalescent process}
Unlike the previous case (drastic bottleneck), we start by analysing the (backward in time) model for the genealogy of a sample as the
population grows to infinity. This will then provide a convenient description
of the (forwards in time) impact of a bottleneck. In this subsection we
measure time backwards; that is by generation $g\in \N_0$ we mean $g$ steps back in the past starting from the present generation $0$. 

For any $n\in \N$ let $\mathcal{P}_n$ be the set of partitions of $\{1,\ldots,n\}$, and for $\pi \in \mathcal{P}_n$ let $\abs{\pi}$ be the number of blocks in the partition $\pi$. To incorporate the spatial structure of the model, we introduce the following space of \emph{marked} partitions 
\begin{equation*}
	\mathcal{P}_n^{\demes}:=\{\xi=(\pi,\vec{u})\,|\, \pi\in \mathcal{P}_n, \vec{u}\in\{1,\ldots,\demes\}^{|\pi|}\}.
\end{equation*}
For example, for $k=8$ and $D=3$, a marked partition in $\mathcal{P}_8^{3}$ could be \[\xi=\{\{1\}^2,\{3\}^1,\{4\}^1,\{2,5,7\}^3,\{6,8\}^1\},\]
where the superscripts denote the marks. 
Fix $N\in \N$ and consider the Wright--Fisher model with soft bottlenecks characterised by \ref{senumerate:1}, \ref{senumerate:2} and \ref{senumerate:3}. We start by taking a sample of $n\ll N$ individuals in the present generation and we describe it as $\xi_0=(\pi_0,\underline{u}_0)\in\mathcal{P}_n^D$ such that 
\begin{equation*}
\pi_0=\{\{1\},\{2\},\ldots ,\{n\}\}.
\end{equation*}
From this we construct the \emph{ancestral line process} $\{\Pi^N(g)\}_{g \in \N_0 }$ as follows. The partitions evolve according to the following equivalence relation
\begin{equation*}
i\sim_g	j \quad \quad \text{if individuals $i$ and $j$ find a common ancestor within $g$ generations in the past.}
\end{equation*}
To each block of the partition we assign a mark that denotes the location of the corresponding ancestor alive at time $g$ in the past. The possible transitions of this process are mergers of blocks whenever some individuals find a common ancestor or changes of marks due to migrations. The aim of this section is to investigate the limit of this partition-valued process as $N \to \infty$.

We introduce some notation to describe the possible transitions of the genealogy of a sample, we write
 \begin{equation*}
 \begin{split}
 &\xi \succi \eta	\hspace{0.65cm} \text{if the partition } \eta \text{ is constructed by merging two $i$-blocks of } \xi, \\
 &\xi \simi \eta	\hspace{0.4cm} \text{if the partition } \eta \text{ is constructed by relabelling exactly one $i$-block of } \xi \text{ as a $j$-block},\\
  &\xi \succsimi \eta	\hspace{0.65cm} \text{if the partition } \eta \text{ is constructed from } \xi \text{ by merging $i$-blocks and relabelling}\\
  & \hspace{1.59cm} \text{$i$-blocks only}.\\
   \end{split}
 \end{equation*}
The first transition is a binary merger within a particular deme, while the second indicates a single migration from deme $i$ to deme $j$. The last describes multiple mergers and migrations involving deme $i$ only, i.e.\ ancestral lines only merge in that subpopulation and there are migrations starting only from there. For example, for $k=8$ and $D=3$, if \[\xi=\{\{1\}^2,\{3\}^1,\{4\}^1,\{2,5,7\}^3,\{6,8\}^1\}\text{ and } \eta=\{\{1\}^2,\{3,4\}^1,\{2,5,7\}^3,\{6,8\}^2\},\] we have $\xi \Succsim{1} \eta$. We will prove that these are the only three transitions that occur in the stochastic process describing the limiting genealogy. 

We now introduce a coalescent process which we will need in the definition of the candidate limit of the genealogy of a sample.
\begin{definition}\label{def:bottleneck-coalescents}
Fix $n \geq 1$ and $i \in \{1,\ldots ,\demes\}$, let $(\pii(t))_{t \geq 0}$ be a Markov process on $\mathcal{P}_n^{\demes}$ with rate matrix $\tilde{Q}^{\text{\tiny{($i$)}}}=\{\tilde{q}_{\xi\eta}^{\text{\tiny{($i$)}}}\}$ such that

\begin{equation*}\arraycolsep=3pt\def\arraystretch{1.9} \tilde{q}_{\xi\eta}^{\text{\tiny{($i$)}}} =   \left\{ \begin{array}{lll}
         1& \text{if} \hspace{0.2cm}&  \xi \succi \eta, \\
        c_{ji}& \text{if} \hspace{0.2cm}& \xi \simi \eta.
            \end{array} \right. 
        \end{equation*}
\end{definition}
The processes $(\pii(t))_{t \geq 0}$ is a structured Kingman coalescent presenting binary mergers only in subpopulation $i$ and migration only from subpopulation $i$. It describes the limiting genealogy of a soft bottleneck event on island $i$ on the $l^N$ time scale, where $l^N$ is the duration of the bottleneck itself. As this duration is much smaller than the final time rescaling factor $N$, a bottleneck will be instantaneous on the time scale of $N$ generations as $N \to \infty$. We will see that its impact on the whole genealogy will consist of running $(\pii(t))_{t \geq 0}$ for $\efflen$ units of time, and making the genealogy jump from the value before the bottleneck to $\pii(\efflen)$. This intuition leads to the following definition of (what we will prove to be) the coalescent describing the genealogy of our process.
\begin{definition}\label{def:soft-coalescent}
Fix $n \geq 1$ and let $(\pii(t))_{t \geq 0}$ be as in Definition \ref{def:bottleneck-coalescents}. We define the \emph{soft-bottleneck structured} $n$-coalescent 
as the Markov process $(\Pi(t))_{t\geq 0}$ on the set $\mathcal{P}_n^{\demes}$ with rate matrix $Q=\{q_{\xi\eta}\}$ such that
 \begin{equation*}\arraycolsep=3pt\def\arraystretch{1.9} q_{\xi\eta} =   \left\{ \begin{array}{llllll}
         \frac{1}{\omega_i}+\frac{\gamma}{\demes}\P(\pii(\efflen)=\eta\,|\,\pii(0)=\xi)& \text{if} \hspace{0.2cm}& \xi \succi \eta, \\
        \frac{c_{ji}}{\omega_i}+\frac{\gamma}{\demes}\P(\pii(\efflen)=\eta\,|\,\pii(0)=\xi)& \text{if} \hspace{0.2cm}& \xi \simi \eta,\\
        \frac{\gamma}{\demes}\P(\pii(\efflen)=\eta\,|\,\pii(0)=\xi)& \text{if} \hspace{0.2cm}& \xi \succsimi \eta.
            \end{array} \right. 
        \end{equation*}
\end{definition}
Similarly to the frequency process in the previous section, $\{\Pi^N(g)\}_{g \in \N_0 }$ does not satisfy the Markov property due to the presence of bottlenecks and we use the topology induced by the metric $d_\lambda$ defined in Definition \ref{def:dmu} to obtain the convergence to the limiting coalescent $(\Pi(t))_{t\geq 0}$.
We can indeed generalise Definition \ref{def:dmu} to the space $\mathcal{D}([0,T];\mathcal{P}_n^\demes)$ of $\mathcal{P}_n^{\demes}$-valued c\`adl\`ag functions on $[0,T]$.
Firstly, we endow the space $\mathcal{P}_n^\demes$ with the discrete metric $\varrho(\xi,\eta):=1-\delta_{\xi,\eta}$ for all $\xi,\eta \in \mathcal{P}_n^\demes$, where $\delta_{\xi,\eta}$ denotes the Kronecker symbol; $\mathcal{P}_n^\demes$ is then a complete and separable metric space.
For any $x_1,x_2\in \mathcal{D}([0,T];\mathcal{P}_n^\demes)$, let $\overline{d}_{\lambda}$ be the semi-metric
\begin{equation}\label{eq:overline-d-lambda}
\begin{split}
&\overline{d}_{\lambda}(x_1,x_2) := \\ &\inf_{A\in \mathcal{I},f\in \mathcal{F}}\left \{\int_0^\infty e^{-u}d(x_1,x_2,f,A,u)du \,\lor \, \norm{\text{Id}-f}_{\infty}\, \lor \, \lambda([0,T]\setminus A) \, \lor \, \varrho\big(x_1(T),x_2(T)\big)\right \},	
\end{split}
\end{equation}
where $\norm{.}_{\infty}$ is the uniform norm and 
\begin{equation*}
d(x_1,x_2,f,A,u)=\sup_{t\in A}\varrho\big(x_1(t \land u),x_2(f(t)\land u)\big).
\end{equation*}
The metric $d_\lambda$ can be then defined as 
\begin{equation*}
d_{\lambda}(x_1,x_2)=\frac{\overline{d}_{\lambda}(x_1,x_2)+\overline{d}_{\lambda}(x_2,x_1)}{2},
\end{equation*}
for any $x_1,x_2\in \mathcal{D}([0,T];\mathcal{P}_n^\demes)$.
We are now ready to state the convergence theorem.
\begin{theorem}\label{thm:soft-convergence-coalescent}
Fix $N\in \N$, $\gamma>0$, and $n \geq 1$, fix also a sequence $\{l_k^N\}_{k\in \N}$, a sequence $\{s_k^N\}_{k\in \N}$ and a function $\presize:\N \rightarrow \N$ satisfying \ref{senumerate:1}, \ref{senumerate:2}, \ref{senumerate:3} respectively. Let $\{\Pi^N(g)\}_{g\in \N_0}$ be the ancestral line process of the structured Wright--Fisher model with soft bottlenecks parametrised by $N$, $\gamma$, $n$, $\{l_k^N\}_{k\in \N}$, $\{s_k^N\}_{k\in \N}$ and $\{\presize_k\}_{k\in \N}$. Let $(\Pi(t))_{0 \leq t \leq T}$ be the process defined in Definition \ref{def:soft-coalescent} and assume that $\Pi^N(0)\xRightarrow[N\rightarrow \infty]{} \Pi(0)$ weakly. Then for any $T>0$
\begin{equation*}
(\Pi^N_{\lfloor Nt\rfloor})_{0 \leq t \leq T} \xRightarrow[N\rightarrow \infty]{d_{\lambda}} (\Pi(t))_{0 \leq t \leq T}\,,
\end{equation*}
where $\xRightarrow{d_{\lambda}}$ denotes the weak convergence in the topology induced by $d_{\lambda}$.
\end{theorem}

\begin{proof}[Proof of Theorem \ref{thm:soft-convergence-coalescent}]
As in the proof of Theorem \ref{thm:drastic-diffusion}, we introduce a time-homogeneous Markov process that we prove to be close, with respect to $d_\lambda$, to the ancestral line process and converging to $\Pi$.
 Let $k_g$ be the $g$-th generation that is not in the set $\cup_{k=1}^{\infty}\{\underline{t}^N_k+1,\ldots ,\overline{t}^N_k-1\}$, and define the process $\{V^N(g)\}_{g \in \N_0}$ as\begin{equation*}
V^N(g):=\Pi^N(k_g).
\end{equation*}
In other words, $V^N$ sees the bottleneck as lasting only one generation (the last one, backward in time). 
To prove that the distance between $V^N$ and $\Pi^N$ is small in probability, we first have to show the a.s.\ absence of bottlenecks in the proximity of the time horizon.

\noindent \textbf{Step 1.} Again, we proceed as in the first step of the proof of Theorem \ref{thm:drastic-diffusion}. Let $m^N_T$ be the index of the last bottleneck that ends before generation $NT-1$, i.e.\ \mbox{$m^N_T:=\sup\{k:\overline{t}^N_k < NT-1\}$}, and set $\delta^N:=\sum_{k=1}^{m^N_T}(l^N_k-1)$. We want to prove that a.s.\ there is no bottleneck in the time interval $[NT-1,NT+\delta^N]$ by showing that the probability of the complementary event \[\mathfrak{B}_N:=\big\{\omega \in \Omega\, :\, \exists k\in \N \text{ s.t.\ } \{NT-1,NT\ldots ,NT+\delta^N\} \cap \{\underline{t}^N_k+1,\underline{t}^N_k+2,\ldots ,\overline{t}^N_k\}\neq \varnothing  \big\},\] tends to zero as $N \to \infty$. 
We define
\[\mathfrak{B}_{N,k}:=\big\{\omega \in \Omega\, : \{NT-1,NT,\ldots ,NT+\delta^N\} \cap \{\underline{t}^N_k+1,\underline{t}^N_k+2,\ldots ,\overline{t}^N_k\}\neq \varnothing  \big\},\]
then we can follow the same computations as in the first step of proof of Theorem \ref{thm:drastic-diffusion} and get for any $\varepsilon>0$
\begin{align}
&\P\left(\mathfrak{B}_N\right)=\sum_{m\in \N}\sum_{k=2}^{NT-1}\P\left(\mathfrak{B}_N\,,\,m^N_T=m\,,\,\overline{t}^N_{m^T_N}=NT-k\right)\notag\\
&=\sum_{m\in \N}\sum_{k=2}^{NT-1}\P\left(\mathfrak{B}_{N,m+1},\,m^N_T=m\,,\,\overline{t}^N_{m}=NT-k \right) \notag\\
&\leq \sum _{m=1}^{K_\varepsilon}\sum_{k=2}^{NT-1}\P\left(\mathfrak{B}_{N,m+1},\,m^N_T=m\,,\,\overline{t}^N_{m}=NT-k \right) + \P(m^N_T\geq K_\varepsilon)\notag \\
&\leq \sum _{m=1}^{K_\varepsilon}\sum_{k=2}^{NT-1}\P\left(\mathfrak{B}_{N,m+1},\,m^N_T=m\,,\,\overline{t}^N_{m}=NT-k \right)\label{eq:NT2-bott2-1}\\
&\quad + \P(m^N_T\geq K_\varepsilon)\label{eq:NT2-bott2-2},
\end{align}
where there exist $K_\varepsilon,N_\varepsilon\in \N$ such that for all $N\geq N_\varepsilon$
\begin{equation*}
\P(m^N_T\geq K_\varepsilon)<\frac{\epsilon}{2}.	
 \end{equation*}
We now denote by $\delta^N_m=\sum_{k=1}^{m}(l^N_k-1)$. For \eqref{eq:NT2-bott2-1}, by Assumption \ref{senumerate:1}, we get there exists $\overline{N}_\varepsilon$ such that for all $N \geq \overline{N}_\varepsilon$
\begin{equation*}
\begin{split}
&\sum _{m=1}^{K_\varepsilon}\sum_{k=2}^{NT-1}\P\left(\mathfrak{B}_{N,m+1},\,m^N_T=m\,,\,\overline{t}^N_{m}=NT-k \right)\\
&=\sum _{m=1}^{K_\varepsilon}\sum_{k=2}^{NT-1}\P\left(s^N_{m+1}<k+\delta^N_m\,,\, s^N_{m+1}+l^N_{m+1}\geq k-1\,,\,\overline{t}^N_{m}=NT-k \right)\\
&=\sum _{m=1}^{K_\varepsilon}\sum_{k=2}^{NT-1}\P\left(k-l^N_{m+1}-1 \leq s^N_{m+1}<k+\delta^N_m\,,\,\overline{t}^N_{m}=NT-k \right)\\
&\leq \sum _{m=1}^{K_\varepsilon}\sum_{k=2}^{NT-1}\P\left(k-l^N_{m+1}-1 \leq s^N_{m+1}<k+\delta^N_m\right)\P\left(\overline{t}^N_{m}=NT-k \right).
\end{split}
\end{equation*}
Since the variables $s^N_k$ are geometrically-distributed with parameter $\gamma/N$, the above expression is equal to
\begin{equation*}
    \begin{split}
        &\sum _{m=1}^{K_\varepsilon}\sum_{k=2}^{NT-1}\sum_{s=k-l^N_{m+1}-1}^{k+\delta^N_m-1}\frac{\gamma}{N}\left(1-\frac{\gamma}{N}\right)^{s-1}\P\left(\overline{t}^N_{m}=NT-k \right)\\
&\leq \sum _{m=1}^{K_\varepsilon} \frac{\gamma}{N}\left(l^N_{m+1}+\delta^N_m\right)\leq \frac{\gamma}{N}2\,K_\varepsilon\left(\sum_{j=1}^{K_{\varepsilon}+1}l^N_j\right)<\frac {\varepsilon}{2}.
    \end{split}
\end{equation*}
Putting the two inequalities together we have $\P(\mathfrak{B}_N)\longrightarrow \, 0$ $\text{as }N\to \infty.$

\noindent \textbf{Step 2.}
We prove that the distance $d_\lambda$ between $\{\Pi^N(g)\}_{g \in \N_0}$ and $\{V^N(g)\}_{g \in \N_0}$ is small in probability by defining an intermediate process $Z^N$ with
\begin{equation*}
Z^N(g):=\Pi^N(\pi^N(g)),
\end{equation*}
where $\pi^N$ is the random projection $\pi^N(0)=0$ and $\pi^N(g)=\max\{k \in \N_0: k \leq g, k \notin \cup_{m=1}^{\infty}\{\underline{t}^N_m+1,\ldots ,\overline{t}^N_m-1\}\}$.

We start by bounding  $d_{\lambda}\big((V^N(\Nt))_{0\leq t \leq T},(Z^N(\Nt))_{0\leq t \leq T}\big)$. To this end, we consider the time change $f^N$ defined in \eqref{eq:time-change}.
To bound $\overline{d}_{\lambda}(V^N(\Nt),Z^N(\Nt))$, see \eqref{eq:overline-d-lambda}, we choose $A=[0,T)$ and $f=(f^N)^{-1}$ and for $\overline{d}_{\lambda}(Z^N(\Nt),V^N(\Nt))$ we choose $A=[0,T)$ and $f=f^N$; thus
\begin{equation}\label{eq:dist-V-Z1}
\begin{split}
&d_{\lambda}\big((V^N(\Nt))_{0\leq t \leq T},(Z^N(\Nt))_{0\leq t \leq T}\big)\leq \norm{V^N(\lfloor N\cdot\rfloor)-Z^N \big( \lfloor N(f^N)^{-1}(\cdot)\rfloor \big)}_{\infty}\\&+\norm{\text{Id}-(f^N)^{-1}}_{\infty}+\norm{V^N(NT)-Z^N\big(N(f^N)^{-1}(T)\big)}.	
\end{split}
\end{equation}
Since we have that 
\begin{equation*}
V^N(\Nt)=Z^N(\lfloor N(f^N)^{-1}(t)\rfloor) \quad \forall t\leq T-(\delta^N-1)/N,
\end{equation*}
to compute the uniform norm in \eqref{eq:dist-V-Z1}, we can focus just on the interval $[T-\delta^N/N,T)$ in which $Z^N$ is constant and equal to $V^N(N(T-\delta^N/N))$. Conditioning on $\mathfrak{B}_N^c$, we have
\begin{equation*}
\begin{split}
&\P\left(\norm{V^N(\lfloor N\cdot\rfloor)-Z^N\big(\lfloor N(f^N)^{-1}(\cdot)\rfloor \big)}_{\infty}>\epsilon\right)\\&\leq \P\left(\sup_{t\in [0,1)}\left\{V^N(\lfloor N(T-t\delta^N/N)\rfloor)-V^N(N(T-\delta^N/N))\right\}>\epsilon \right),
\end{split}
\end{equation*}
which tends to zero in probability as $N \to \infty$ since, in absence of bottlenecks, the probability that the ancestral line process changes in $\delta_N$ generations tends to zero in probability. Indeed any possible transition happens with a probability of order $O(N^{-1})$ or smaller and  $\delta_N/N\to 0$ in probability. Hence $\norm{V^N(\lfloor N\cdot\rfloor)-Z^N\big(\lfloor N(f^N)^{-1}(\cdot)\rfloor \big)}_{\infty}\to 0$ in probability. For the same reason, $\norm{V^N(NT)-Z^N\big(N(f^N)^{-1}(T)\big)}$ also converges to zero in probability. Finally, we have
\begin{equation*}
\norm{\text{Id}-(f^N)^{-1}}_{\infty}=\frac{\delta^N}{N}\xrightarrow{N\rightarrow \infty}0 \quad  \text{in probability},
\end{equation*}
so that we can conclude 
\begin{equation*}
d_{\lambda}\big((V^N(\Nt))_{0\leq t \leq T},(Z^N(\Nt))_{0\leq t \leq T}\big)\xrightarrow{N\rightarrow \infty}0 \quad  \text{in probability}.
\end{equation*}
To prove that also $d_{\lambda}\big((\Pi^N(\Nt))_{0\leq t \leq T},(Z^N(\Nt))_{0\leq t \leq T}\big)$ converges to zero in probability, we can follow the same steps as in the proof of Theorem \ref{thm:drastic-diffusion}.
Hence we can conclude
\begin{equation*}
d_{\lambda}\big((\Pi^N(\Nt))_{0\leq t \leq T},(V^N(\Nt))_{0\leq t \leq T}\big)\xrightarrow{N\rightarrow \infty} 0 \,\,\,\,\,\, \text{in probability.}	
\end{equation*}
\vskip 11pt
\noindent\textbf{Step 3.} 
It remains to prove that the generator of $\{V(i)^N\}_{i \in \N_0}$ converges to the infinitesimal generator of $(\Pi(t))_{t \geq 0}$. Provided this is true, since the limit is Feller continuous, weak convergence follows by \cite[Theorem 19.25 and 19.28]{Kallenberg}. We can split the generator $G^N$ of $\{V(i)^N\}_{i \in \N_0}$ by conditioning with respect to the event that there is a bottleneck after a unit of time and the location of this bottleneck. We then get
\begin{equation*}
\begin{split}
G^Nh(\xi)=N\E_{\xi}[h(V^N_1)-h(\xi)]={}&N\E_{\xi}[h(V^N_1)-h(\xi)| s^N_1>0]\,\P(s^N_1>0)\,\\&+\sum_{i=1}^\demes \frac{N}{\demes}\E_{\xi}[h(V^N_1)-h(\xi)|, s^N_1=0, \beta^N_1=i]\,\P(s^N_1=0).
\end{split}
\end{equation*}
The first term converges uniformly to the generator of the structured Kingman coalescent \cite{Herbots_1998}, that is
\begin{equation*}
\begin{split}
\sum_{i=1}^\demes\sum_{\eta\in \mathcal{P}_k^{\demes}}\left[ \frac{1}{\omega_i}\mathds{1}_{\big\{\xi \succi \eta\big\}}+\sum_{j \neq i}\frac{c_{ji}}{\omega_i}\1_{\big\{\xi \simi \eta\big\}}\right]\big(h(\eta)-h(\xi)\big).
\end{split}
	\end{equation*}

Since $P(s^N_1=0)=\gamma/N$, the second term can be written as
\begin{equation*}
\begin{split}
&\sum_{i=1}^\demes \frac{\gamma}{\demes}\E_{\xi}[h(V^N_1)-h(\xi)|, s^N_1=0, \beta^N_1=i]=\sum_{i=1}^\demes \frac{\gamma}{\demes}\sum_{\eta\in \mathcal{P}_k^{\demes}}\big(h(\eta)-h(\xi) \big)\P_{\xi}\left(\prepii(l^N)=\eta\right),
	\end{split}
	\end{equation*}
where $\{\prepii(g)\}_{g \in \N_0}$ is defined in \ref{def:pi-tilda} and simply describes the restriction of the discrete ancestral line process to one soft bottleneck affecting deme $i$. Using Proposition \ref{prop:appendix-coalescent-convergence} and \eqref{eq:lambda} we have
\begin{equation*}
\begin{split}
&\P_{\xi}\left(\prepii(l^N)=\eta\right)\xrightarrow{N\rightarrow \infty}\P_{\xi}\left(\pii(\efflen)=\eta\right)	.
\end{split}
\end{equation*}
Putting everything together we conclude that $G^Nh\xrightarrow{N\rightarrow \infty}G\,h$ uniformly, where
\begin{equation*}
\begin{split}
Gh(\xi)=&\sum_{i=1}^\demes\sum_{\eta\in \mathcal{P}_k^{\demes}}\left[ \frac{1}{\omega_i}\mathds{1}_{\big\{\xi \succi \eta\big\}}+\sum_{j \neq i}\frac{c_{ji}}{\omega_i}\1_{\big\{\xi \simi \eta\big\}}\right]\big(h(\eta)-h(\xi)\big)	\\&+\sum_{i=1}^\demes \frac{\gamma}{\demes}\sum_{\eta\in \mathcal{P}_k^{\demes}}\big(h(\eta)-h(\xi) \big)\P_{\xi}\left(\pii(\efflen)=\eta\right)
\end{split}
\end{equation*}
is the generator of the process defined in Definition~\ref{def:soft-coalescent}.
\end{proof}

\subsection{The frequency process}
We now study the evolution of the frequency of type-$a$ individuals. We denote the discrete frequency process by $\{\mathbf{X}^{N}(g)\}_{g \in \N_0}=\{(X^{N}_1(g),\ldots,X^{N}_\demes(g))\}_{g \in \N_0}$, and, in Theorem \ref{thm:soft-diffusion} we prove that, after rescaling time by a factor $N$, it converges to a diffusion process with jumps. The main difference from the drastic bottlenecks model is that both the length and the size of the bottlenecks go to infinity with $N$. This means that, even if the length $l^N$ is negligible after rescaling time by a factor $N$, we will need an intermediate rescaling to describe the bottleneck's effect. The intuition is that since, backward in time, a bottleneck can be seen as a Kingman coalescent with emigration running for $\efflen$ units of time, we can use the same process to describe the change in the frequency of type-$a$ individuals.

We start by introducing an alternative description of the ancestral process describing a soft bottleneck affecting deme $i$, $(\pii(t))_{t \geq 0}$, introduced in Definition~\ref{def:bottleneck-coalescents}, following the model presented in~\cite{DT_1987}.
We use a different strategy from the usual Kingman coalescent: instead of keeping track of the equivalence classes representing the ancestral relations, we describe the genealogy using the frequency of the initial population in each of those equivalence classes. We explain this procedure in what follows.

Let the process start at the present time $0$ and consider a time $t>0$ in the past.
A certain proportion $\zeta_1$ of the initial population (population at time $0$) may be in deme $i$ and share a common ancestor.
A further proportion $\zeta_2$ may be in deme $i$ as well but share a different ancestor.
Suppose there are $k$ of these groups in the $i$-th island at time $t$, each one including a proportion $\zeta_1,\ldots,\zeta_k$ of the initial population.
We then recall that an additional part of the initial population is descended from migrants arriving from other demes between time $t$ and $0$. For example, suppose that, at a certain time $s\in(0,t)$ in the past, a proportion $z_r$ of the initial population sharing a common ancestor migrates to deme $j$. As we don't have simultaneous migrations in $(\pii(t))_{t \geq 0}$, each class that migrates to deme $j$ does so at a different time, hence we can order these migrating classes from oldest to newest as $\mathbf{Z}^{\j}=(z_1^\j,z_2^\j,\ldots)$, where we measure time backwards. 

To sum up, we can write the genealogical information at time $t$ as:
\begin{equation}\label{eq:freq_state}
(k;\zeta_1,\zeta_2,\ldots,\zeta_k;\mathbf{Z}^{\loc{1}},\ldots,\mathbf{Z}^\loc{i-1},\mathbf{Z}^\loc{i+1},\ldots ,\mathbf{Z}^\loc{\demes}),	
\end{equation}
for some $k\geq 0$, $0\leq \zeta_r \leq 1$, $0\leq z^\j_r \leq 1$, and $\sum_{r=1}^{k}\zeta_i+\sum_{j=1}^\demes \sum_{r=1}^\infty z_r^\j=1$. 
Once the process is at a state of the form \eqref{eq:freq_state}, as $t$ increases, only two events are possible: two classes in deme $i$ merge, and this happens at rate $\binom{k}{2} $, or one class migrates to another island, at rate $kc_i$, where we recall that $c_i = \sum_{j\neq i} c_{ji}$ is the number of migrants. In each case, the number of classes in deme $i$ decreases by one. The time until the next change occurs is exponentially distributed with parameter $\frac{k(k-1+2c_i)}{2}$. Given that such a change occurs, it will be the merging of a particular pair of the $k$ individuals with probability $\frac{2}{k(k-1+2c_i)}$, while, with probability $\frac{2c_{ji}}{k(k-1+2c)}$ it will be the migration of an individual to deme $j\neq i$. 

We now formalise this construction: define a pure death process $(\mathcal{D}^\i(t))_{t\geq 0}$ with death rate $\frac{k(k-1+2c_i)}{2}$ from state $k$ and initial condition $\mathcal{D}^\i(0)=\infty$. Since the time $\mathcal{D}^\i(t)$ takes to reach any state $m\in\N_0$ has finite mean $\sum_{k>m}\frac{2}{k(k-1+2c_i)}$, $\mathcal{D}^\i(t)$ is almost surely finite and behaves as an ordinary death process for $t>0$, \cite{DT_1987}. 

Next, we define a Markov chain with state space $\N_{\infty}\times \Delta$, where
\begin{equation*}
\begin{split}
\Delta=\Bigg\{\delta=&(\mathbf{Z}^k;\mathbf{Z}^{\loc{1}},\ldots,\mathbf{Z}^\loc{i-1},\mathbf{Z}^\loc{i+1},\ldots,\mathbf{Z}^\loc{\demes})|\, \mathbf{Z}^k=(\zeta_1,\zeta_2,\ldots,\zeta_k),  \,\text{$k\geq 0$, $0\leq \zeta_r \leq 1$,}\\& \text{$\mathbf{Z}^{\loc{j}}=(z_1^\j,z_2^\j,\ldots)$,}  \text{ $0\leq z^\j_r \leq 1$,} \text{ and $\sum_{i=1}^{k}\zeta_i+\sum_{j\neq i} \sum_{r=1}^\infty z_r^\j=1$}\Bigg\},
\end{split}
\end{equation*}
and 
\begin{equation*}
\mathbf{Z}^k_{-j}:=(\zeta_1,\ldots,\zeta_{j-1},\zeta_{j+1},\ldots\zeta_k).
\end{equation*}
Let $\{\mathfrak{N}^\i(g)\}_{g\in \N_0}$ be a discrete-time Markov chain on $\N_{\infty}\times \Delta$, independent of $(\mathcal{D}^\i(t))_{t\geq 0}$, and with the following transition probabilities from a state of the form\\ $\mathfrak{N}^\i_k=(k;\zeta_1,\ldots,\zeta_k;\mathbf{Z}^{\loc{1}},\mathbf{Z}^\loc{2},\ldots,\mathbf{Z}^\loc{\demes})$. For any $r,l \in\{1,\ldots ,k\}$ and $j \in \{1,\ldots ,\demes\}$
\begin{align*}
        \P\Bigg(\mathfrak{N}^\i(k-1)=\Big(k-1;\mathbf{Z}^k_{-r}+\zeta_re_l,;\mathbf{Z}^{\loc{1}},\mathbf{Z}^\loc{2},\ldots,\mathbf{Z}^\loc{\demes}\Big)\big|\mathfrak{N}^\i_k\Bigg)&=\frac{1}{k(k-1+2c_i)},\\
        \P\Bigg(\mathfrak{N}^\i(k-1)=\Big(k-1;\mathbf{Z}^k_{-r};\mathbf{Z}^{\loc{1}},\ldots,\mathbf{Z}^{\loc{j-1}},(\zeta_r,z_1^\j,z_2^\j,\ldots),&\mathbf{Z}^{\loc{j+1}},\mathbf{Z}^{\loc{\demes}} \Big)\big|\mathfrak{N}^\i_k\Bigg)\\
        \phantom{\P\Big(\mathfrak{N}^\i(k-1)=(k-1;\mathbf{Z}^k_{-i}; \mathbf{Z}^{\loc{1}},\ldots,\mathbf{Z}^{\loc{j-1}},(\zeta_i,z_1^\j,z_2^\j,\ldots),}
        &=\frac{2c_{ji}}{k(k-1+2c_i)}.
\end{align*}
We then define the process $(\mathcal{N}^\i(t))_{t \geq 0}$ describing the fraction of the time-$0$ population descended from each common ancestor as
\begin{equation}\label{eq:backward-construction}
\mathcal{N}^\i(0)=(\infty;0,0,\ldots) \hspace{1cm} \text{and} \hspace{1cm} \mathcal{N}^\i(t)=\mathfrak{N}^\i(\mathcal{D}^\i(t))\,\,\,\forall t>0.
\end{equation}
The distribution of $\{\mathcal{N}^\i(t)\}_{t \geq 0}$ can be found in \cite[Section 3]{DT_1987}, where the authors first compute the distribution of $\mathcal{D}^\i(t)$ and then derive the conditional distribution of $\mathcal{N}^\i(t)$ given $\mathcal{D}^\i(t)=k$. For any subset $A$ of $\Delta$, we denote
\begin{equation}\label{eq:nu-def}
\nu^{k,\i}(A)=\P(\mathfrak{N}^\i(k)\in A).
\end{equation}
Then,
\begin{equation*}
\P(\mathcal{N}^\i(t)\in \{k\}\times A)=\P(\mathcal{D}^\i(t)=k)\nu^{k,\i}(A).
\end{equation*}
The above description allows us to write down the \emph{candidate} frequency after (forward in time) a bottleneck affecting deme $i$, given that the allele frequencies right before the event were  $x_1,\ldots,x_\demes\in[0,1]$,
\begin{equation*}
\sum_{r=1}^{k}\zeta_r\,\text{B}_r^{x_i}+\sum_{j \neq i}\sum_{r\geq 1}z_r^{\j}\,\text{B}_r^{x_j},	
\end{equation*}
where the $B_i$'s are independent Bernoulli random variables with respective parameters $x_i$, and $k$ is the number of lines left in deme $i$ at time $\efflen$, i.e.\ $\mathcal{D}^\i(\efflen)=k$.
The intuition is that each fraction $\zeta_r$ of the population at the end of the bottleneck shares one ancestor in deme $i$ that chooses a type-$a$ parent with probability $x_i$, and each fraction $z_r^{\j}$ shares one ancestor in deme $j$ that chooses a type-$a$ parent with probability $x_j$.
\begin{theorem}\label{thm:soft-diffusion}
Fix $N\in \N$, $\gamma>0$, and $n \geq 1$, fix also a sequence $\{l_k^N\}_{k\in \N}$, a sequence $\{s_k^N\}_{k\in \N}$ and a function $\presize:\N \rightarrow \N$ satisfying \ref{senumerate:1}, \ref{senumerate:2}, \ref{senumerate:3} respectively. Let $\{\mathbf{X}^N(g)\}_{g\in \N_0}$ be the frequency process of the structured Wright--Fisher model with soft bottlenecks parametrised by $N$, $\gamma$, $\{l_k^N\}_{k\in \N}$ and $\{\presize_k\}_{k\in \N}$. Assume $\mathbf{X}^N(0)\xRightarrow[N\rightarrow \infty]{} \mathbf{X}(0)$ weakly, then for any $T>0$
\begin{equation*}
(\mathbf{X}^N(\lfloor Nt\rfloor))_{0 \leq t \leq T} \xRightarrow[N\rightarrow \infty]{d_{\lambda}} (\mathbf{X}(t))_{0 \leq t \leq T}\,,
\end{equation*}
where $\xRightarrow{d_{\lambda}}$ denotes the weak convergence in the topology induced by $d_{\lambda}$ and $(\mathbf{X}(t))_{t \geq 0}$ is the unique strong solution of the following system of SDEs

\begin{align}\label{eq:soft-frequency}
&dX_i(t)=\sum_{j\neq i}\frac{c_{ji}}{\omega_i}(X_j(t)-X_i(t))dt+\sqrt{\frac{1}{\omega_i}X_i(t)(1-X_i(t))}dB_i(t)\\
&\phantom{=}+\frac{\gamma}{\demes}\int_{\mathcal{S}}\left(\Big[\sum_{r=1}^{k}\zeta_r \1_{\{u_{ir}\leq X_i(t^-)\}}+\sum_{j\neq i} \sum_{r=1}^\infty z_r^{\j}\1_{\{u_{jr}\leq X_j(t^-)\}}\Big]-X_i(t^-)\right)\mathcal{N}_i(dt,dk,d\delta,du),\notag
\end{align}
where $(B_i(t))_{t \geq 0}$ are independent Brownian motions, $\mathcal{S}:=\N_0\times\Delta\times [0,1]^{\demes \times \N} $, and $\mathcal{N}_i$ are independent Poisson measures on $\mathcal{S}:=\N_0\times\Delta\times [0,1]^{\demes \times \N}$ with intensity $\gamma dt\otimes\mathfrak{V}^\i(dk,d\delta)\otimes du$. We define $\mathfrak{V}^\i$  as the semidirect product of the probability distribution $\P(\mathcal{D}^\i(\efflen)=\cdot)$, of a pure death process with death rate $\frac{k(k-1+2c_i)}{2\omega_i}$, and the probability kernel $\nu^\i(\cdot,\cdot)$ where $\nu^\i(k,\cdot)$ is the distribution $\nu^{k,\i}(\cdot)$ defined in \eqref{eq:nu-def} and $du$ is the Lebesgue measure on $[0,1]^{\demes \times \N}$. \end{theorem}

\begin{proof}
\textbf{Step 1.} Firstly, as $\{\mathbf{X}^N(g)\}_{g\in \N}$ is not Markovian, we introduce the Markov process
\begin{equation*}
\mathbf{V}^N(k):=\mathbf{X}^N(g_k),
\end{equation*}
where $g_k$ is the $k$-th generation that is not in the set $\cup_{k=1}^{\infty}\{\underline{t}^N_k+1,\ldots ,\overline{t}^N_k-1\}$. Following the steps of the proof of Theorem \ref{thm:drastic-diffusion} and using Assumption \ref{senumerate:1}, we conclude that
\begin{equation*}
d_{\lambda}\big((\mathbf{X}^N(\Nt))_{0\leq t \leq T},(\mathbf{V}^N(\Nt))_{0\leq t \leq T}\big)\xrightarrow{N\rightarrow \infty} 0 \,\,\,\,\,\, \text{in probability}.	
\end{equation*}
\textbf{Step 2.} Now we prove that the generator of $(\mathbf{V}^N(\lfloor Nt\rfloor))_{t \geq 0}$ converges to that of $(\mathbf{X}(t))_{t \geq 0}$ as $N \to \infty$. Weak convergence then follows by \cite[Theorem 19.28]{Kallenberg}. The generator of $\{\mathbf{V}^N(g)\}_{g\in \N_0}$ is, for all $f \in C^3([0,1]^\demes)$,
\begin{equation}\label{eq:discrete-generator-soft-frequency}
\begin{split}
\mathcal{A}^Nf(\mathbf{x})=&N\left(1-\frac{\gamma}{N}\right)\mathcal{L}^Nf(\mathbf{x})
+\sum_{i=1}^\demes \frac{\gamma}{\demes}\E\left[f(\tilde{\mathbf{X}}^{\i,N}(l^N))-f(\mathbf{x})\right],
\end{split}
\end{equation}
where $\mathcal{L}^N= \E_{\mathbf{x}}[f(\mathbf{V}^N_1)-f(\mathbf{x})|s_0^N>1]$ and the Markov chain $\{\tilde{\mathbf{X}}^{\i,N}(g)\}_{g\in \{0,\ldots,l^N\}}$ is the pre-limiting frequency process during a bottleneck in deme $i$, as defined in \eqref{eq:appendix-prelimiting-frequency}. As in the proof of Theorem \ref{thm:drastic-diffusion}, we have that for all $f \in C^3([0,1]^\demes)$, $N\mathcal{L}^Nf \xrightarrow{N\rightarrow \infty} \mathcal{L}f$ where
\begin{equation*}
\mathcal{L}f(\mathbf{x}):= \sum_{i=1}^\demes \left(\sum_{j\neq i}\frac{c_{ji}}{\omega_i}(x_j-x_i)\frac{\partial f}{\partial x_i}(\mathbf{x})+\frac{1}{2\omega_i}x_i(1-x_i)\frac{\partial f^2}{\partial x_i^2}(\mathbf{x})\right).\end{equation*}
Using Proposition \ref{prop:appendix-frequency-convergence} and the fact that 
\begin{equation*}
\int_0^1\frac{1}{\varphi(t)}dt=\efflen,
\end{equation*}
we conclude that the jump part is such that for any $i\in \{1,\ldots,\demes\}$
\begin{equation*}
\E\left[f(\tilde{\mathbf{X}}^{\i,N}(l^N))-f(\mathbf{x})\right]\xrightarrow{N\rightarrow \infty}\E\left[f(\tilde{\mathbf{X}}^{\text{\tiny{($i$)}}}(\efflen))-f(\mathbf{x})\right],
\end{equation*}
where $(\tilde{\mathbf{X}}^\i(t))_{t \geq 0}=(\tilde{X}_1^\i(t),\ldots ,\tilde{X}_D^\i(t))_{t \geq 0}$ is a diffusion process with infinitesimal generator:
\begin{equation*}
\tilde{\mathcal{A}}^\i f(\mathbf{x})=\sum_{j\neq i}c_{ji}(x_j-x_i)\frac{\partial f}{\partial x_i}(\mathbf{x})+\frac{1}{2}x_i(1-x_i)\frac{\partial f^2}{\partial x_i^2}(\mathbf{x}).
\end{equation*}
The generator \eqref{eq:discrete-generator-soft-frequency} therefore converges uniformly, as $N \to \infty$, to
\begin{equation}\label{eq:soft-frequency-generator-secondver}
\begin{split}
& \sum_{i=1}^\demes \left(\sum_{j\neq i}\frac{c_{ji}}{\omega_i}(x_j-x_i)\frac{\partial f}{\partial x_i}(\mathbf{x})+\frac{1}{2\omega_i}x_i(1-x_i)\frac{\partial f^2}{\partial x_i^2}(\mathbf{x})\right)\\
&+\sum_{i=1}^\demes \frac{\gamma}{\demes}\int_{[0,1]}\P_{\mathbf{x}}\left(\tilde{X}_i^{\i}(\efflen)\in dx'\right)\left(f\left(x_1,\ldots ,x_{i-1},x',x_{i+1},\ldots ,x_D\right)-f(\mathbf{x})\right).
\end{split}
\end{equation}
It remains to prove that \eqref{eq:soft-frequency-generator-secondver} is indeed equal to the generator of the limiting process $(\mathbf{X}(t))_{t \geq 0}$ which is
\begin{equation}\label{eq:soft-frequency-generator}
\begin{split}
&\mathcal{A}f(\mathbf{x})= \sum_{i=1}^\demes \Bigg[\Bigg(\sum_{j\neq i}\frac{c_{ji}}{\omega_i}(x_j-x_i)\frac{\partial f}{\partial x_i}(\mathbf{x})+\frac{1}{2\omega_i}x_i(1-x_i)\frac{\partial f^2}{\partial x_i^2}(\mathbf{x})\Bigg)\\&+\frac{\gamma}{\demes}\sum_{k\geq 1}\P(\mathcal{D}^\i(\efflen)=k)\int_{\Delta} \nu^{k,\i}(d\delta)\P\Big((\sum_{r=1}^k\zeta_r B_r^{x_i}+\sum_{j\neq i}\sum_{r \geq k}z^{\j}_rB_r^{x_j})\in dx'\Big)\\&\quad\quad\quad\quad\quad\quad\left(f\left(x_1,\ldots ,x_{i-1},x',x_{i+1},\ldots ,x_D\right)-f(\mathbf{x})\right)\Bigg].
\end{split}
\end{equation}
To conclude we use the duality relation between the forward in time process $(\tilde{X}_i^{\i}(t))_{t \geq 0}$ and the backward in time $(\mathcal{N}^\i(t))_{t \geq 0}$ defined in \eqref{eq:backward-construction}. In particular, we consider the process $\mathcal{N}^\i$ started at $\mathcal{N}^\i(0)=(\infty;0,0,\ldots)$ and we let it run for $\efflen$ units of time. This tells us the number of remaining ancestors at the beginning of the
bottleneck and the frequencies of the initial population descended from each
of them. Each ancestor is of type $a$ with probability $x_j$ if it is in the $j$-th subpopulation, and
the fraction of the current population sharing that ancestor also share its
type. Hence we can write
\begin{equation*}
\begin{split}
&\P_{\mathbf{x}}\left(\tilde{X}_i^{\i}(\efflen)\in dx'\right)\\&=\sum_{k\geq 1}\P(\mathcal{D}^\i(\efflen)=k)\bigintsss_{\Delta} \nu^{k,\i}(d\delta)\P\left(\left(\sum_{r=1}^k\zeta_i B_r^{x_i}+\sum_{j \neq i}\sum_{r \geq k}z^{\j}_rB_r^{x_j}\right)\in dx'\right),
\end{split}
\end{equation*}
and this concludes the proof.
\end{proof}
\begin{remark}
One can prove that \eqref{eq:soft-frequency} has a unique strong solution which is Feller continuous using the technique presented in Proposition \ref{prop:Feller}. 
\end{remark}

\subsection{Duality}
We now prove that the frequency process $(\mathbf{X}(t))_{t \geq 0}$ and the coalescent $(\Pi(t))_{t \geq 0}$ satisfy a duality relation.
We denote by $(\mathbf{N}(t))_{t \geq 0}=(N_1(t),\ldots,N_\demes(t))_{t \geq 0}$ the block-counting process corresponding to the coalescent with \emph{soft-bottleneck} $(\Pi(t))_{t \geq 0}$, i.e.\ $N_i(t)$ is the number of blocks with label $i$ of $\Pi(t)$. Its transitions are:
 \begin{itemize}
\item $\mathbf{n}\mapsto \mathbf{n}-\mathbf{e}_i$ with rate
\begin{equation*}
	\frac{1}{\omega_i}\binom{n_i}{2}+\frac{\gamma}{\demes}\P(\tilde{N}_i^\i(\efflen)=n_i-1,\,\tilde{N}_j^\i(\efflen)=0 \,\,\forall j\neq i),
\end{equation*}
\item $\mathbf{n}\mapsto \mathbf{n}-\mathbf{e}_i+\mathbf{e}_j$ with rate
\begin{equation*}
	\frac{c_{ji}}{\omega_i}n_i+\frac{\gamma}{\demes}\P(\tilde{N}_i^\i(g)=n_i-1, \tilde{N}_j^\i(\efflen)=1),
\end{equation*}
\item $\mathbf{n}\mapsto \mathbf{n}-(n_i-a)\mathbf{e}_i+\sum_{j\neq i}b_j\mathbf{e}_j$ with rate
\begin{equation*}
	\frac{\gamma}{\demes}\P(\tilde{N}_i^{\i,n}(\efflen)=a,\,\tilde{N}_j^\i(\efflen)=b_j \,\,\forall j\neq i),
\end{equation*}
\end{itemize}
where $(\tilde{N}^\i(t))_{t \geq 0}$ is the block-counting process associated to $(\pii(t))_{t \geq 0}$.
For any function $h:\N^\demes \rightarrow \R$, the infinitesimal generator of the block-counting process $\mathbf{N}$ is given by:

\begin{align}\label{eq:soft-coalescent-generator}
&\mathcal{G}h(\mathbf{n}):=\sum_{i=1}^\demes\left[\frac{1}{\omega_i}\binom{n_i}{2}(h(\mathbf{n}-\mathbf{e}_i)-h(\mathbf{n}))+\frac{c_{ji}}{\omega_i}n_i(h(\mathbf{n}-\mathbf{e}_i+\mathbf{e}_j)-h(\mathbf{n}))\right]\\
&+\sum_{i=1}^\demes\frac{\gamma}{\demes}\sum_{\substack{a,(b_j)_{j\neq i}\,:\\ a+\sum_{j\neq i} b_j\leq n_i}}\P(\tilde{N}_i^{\i,n}(\efflen)=a,\,\tilde{N}_j^\i(\efflen)=b_j)\Big(h\big(n-(n_i-a)e_i+\sum_{j\neq i}b_je_j\big)-h(n)\Big).\notag
\end{align}
\begin{theorem}
For every $x\in[0,1]^\demes$ and $n\in \N^\demes$, we have
\begin{equation*}
\E_{x}\Bigg[\prod_{i=1}^\demes X_i(t)^{n_i}\Bigg]=\E_{n}\Bigg[\prod_{i=1}^\demes x_i^{N_i(t)}\Bigg].
\end{equation*}

\end{theorem}
\begin{proof}
Let $\mathcal{A}$ be the generator of the allele-frequency process $(X_t)_{t\geq 0}$, as defined in \eqref{eq:soft-frequency-generator}, and $\mathcal{G}$ the one of the block-counting process $(N_t)_{t\geq 0}$.
Consider the function $h:[0,1]^\demes \times\N^\demes\rightarrow\R$, $h(x,n)=\prod_{i}x_i^{n_i}$. Our aim is to prove that
\begin{equation*}
\mathcal{A}\,h(x,n)=\mathcal{G}\,h(x,n),
\end{equation*}
when $\mathcal{A}$ acts on the first component and $\mathcal{G}$ on the second, and $\mathcal{A}$ is the generator \eqref{eq:soft-frequency-generator} of the allele-frequency process. Since $(X(t))_{t\geq 0}$ is Feller continuous, duality follows from Proposition 1.2 of \cite{JK14}.
We can rewrite $\mathcal{A}h(x,n)$ as in \eqref{eq:soft-frequency-generator-secondver}
\begin{align}\label{eq:long-soft generator}
\mathcal{A}h(x,n)={}& \sum_{i=1}^\demes \left(\sum_{j\neq i}\frac{c_{ji}}{\omega_i}(x_j-x_i)\frac{\partial f}{\partial x_i}(x)+\frac{1}{2\omega_i}x_i(1-x_i)\frac{\partial f^2}{\partial x_i^2}(x)\right)\\
&+\sum_{i=1}^\demes \frac{\gamma}{\demes}\int_{[0,1]}\P_{\mathbf{x}}\left(\tilde{X}_i^{\i}(\efflen)\in dx'\right)\left(f\left(x_1,\ldots ,x_{i-1},x',x_{i+1},\ldots ,x_D\right)-f(\mathbf{x})\right).\notag
\end{align}	
We proved in Theorem \ref{thm:drastic-frequency} that the duality holds for the Wright--Fisher and Kingman part. Following the same computation, we can show the following relation 
\begin{equation*}
\E_{x}\left[\prod_{j=1}^\demes \tilde{X}^\i_j(t)^{n_i}\right]=\E_{n}\left[\prod_{j=1}^\demes x_j^{\tilde{N}^\i(t)}\right] \quad\quad \forall t\geq 0.
\end{equation*}
In fact we have
\begin{equation*} 
\begin{split}
&\tilde{\mathcal{A}}^\i\,h(x,n)\\
&= \frac{1}{2}x_i(1-x_i)n_i(n_i-1)x_i^{n_i-2}\prod_{j\neq i }x_j^{n_j}+\sum_{j \neq i}c_{ji}(x_j-x_i)n_ix^{n_i-1}\prod_{k\neq i}x_k^{n_k}\\
&=\binom{n_i}{2} \Big(x_i^{n_i-1}\prod_{j\neq i }x_j^{n_j}-\prod_{j=1}^\demes x_j^{n_j} \Big) + \sum_{j \neq i}c_{ji}n_i\Big(x_i^{n_i-1}x_j^{n_j+1}\prod_{k\neq i,j }x_k^{n_k}-\prod_{j=1}^\demes x_j^{n_j} \Big)\\
&=\tilde{\mathcal{G}}^\i\,h(x,n),
\end{split}
\end{equation*}
where $\tilde{\mathcal{G}}^\i$ is the generator of $(\tilde{N}^\i(t))_{t \geq 0}$. Hence, for $t=\efflen$, we have 
\begin{equation*}
	\begin{split}
		&\int_{[0,1]}\P_{x,y}\left(\tilde{X}^{\i}(\efflen)\in dx'\right)h\left(x',n\right)=\E_{x}\left[\prod_{j=1}^\demes \tilde{X}^\i_j(\efflen)^{n_i}\right]=\E_{n}\left[\prod_{j=1}^\demes x_i^{\tilde{N}^\i_j(\efflen)}\right]\\&=\sum_{\substack{a,(b_j)\,:\\ a+\sum_{j\neq i} b_j\leq n_i}}\P(\tilde{N}_i^{\i,n}(\efflen)=a,\,\tilde{N}_j^\i(\efflen)=b_j \,\,\forall j\neq i)h\big(x,n-n_ie_i+ae_i+\sum_{j\neq i}b_je_j\big).
	\end{split}
\end{equation*}
Putting this into \eqref{eq:soft-coalescent-generator} and \eqref{eq:long-soft generator} we conclude that
\begin{equation*}
\mathcal{A}h(x,n)=\mathcal{G}h(x,n).
\end{equation*}
\end{proof}

\section{Simulations}\label{section:simulations}

In this final section, we show and discuss some simulations of the coalescent processes introduced in this paper. We are especially interested in the shape of the Site Frequency Spectrum (SFS) associated with these models. The SFS is one of the most widely used summary statistics in population genetics. It describes how frequent certain mutant alleles are within a set of genomic data of a variety of species, and it is often used to compare data with theoretical models. The code is available at \url{https://github.com/martadaipra/bottlenecks-w-migration}.

We start by introducing the \emph{infinite sites model} for genetic variation.  We assume that individuals can mutate, each new mutation occurs at a new site along the chromosome, and that mutations are passed on to all descendants unchanged. In other words, mutations accumulate and do not cancel each other. Given a coalescent tree of a sample of $n$ individuals, we assume that mutations fall on it according to a Poisson point process with constant intensity $\vartheta/2$ per unit branch length, for some $\vartheta>0$. The SFS is then defined as the vector
\begin{equation*}
M(n)=(M_1(n),M_2(n),\ldots,M_{n-1}(n)),	
\end{equation*}
where $M_j(n)$ is the number of sites at which exactly $j$ individuals have a mutation. 

In the following we present some simulations of the logit transform of the mean normalised SFS, and show that we have built a flexible model which can predict very different SFS shapes using a small number of interpretable parameters: the migration rate, the bottleneck rate, and the size and length (resp.\ effective length) for drastic (resp.\ soft) bottlenecks. 
 \begin{figure}
\centering
\begin{subfigure}{.49\textwidth}
  \centering
  \includegraphics[width=1.0\linewidth]{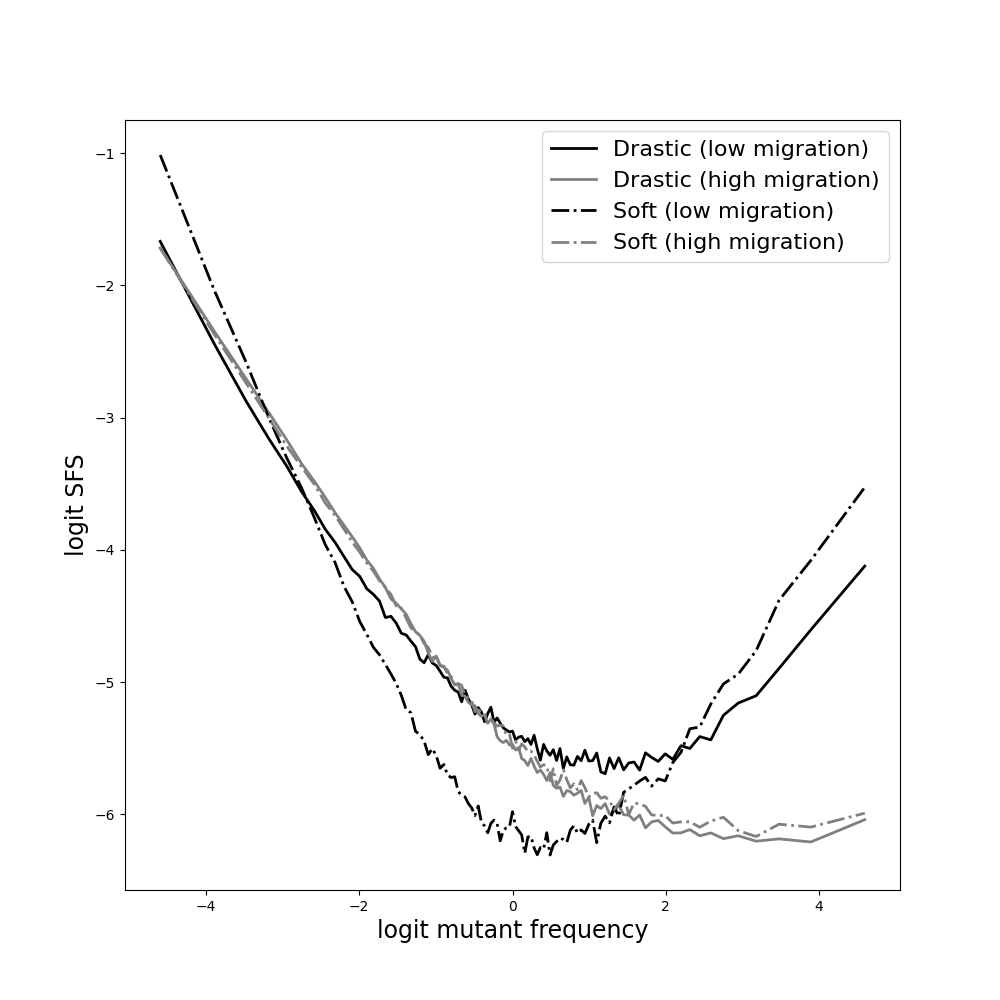}
  \label{fig:sub11}
\end{subfigure}%
\begin{subfigure}{.49\textwidth}
  \centering
  \includegraphics[width=1\linewidth]{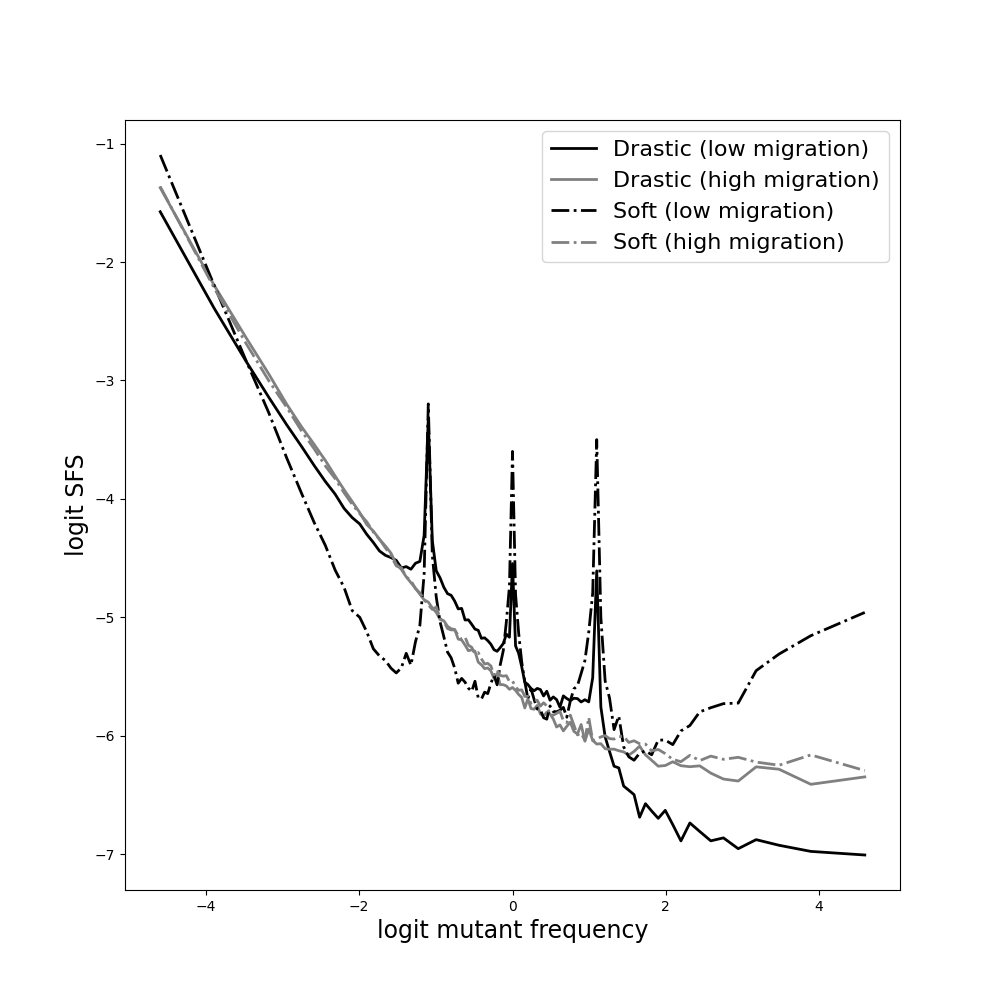}
  \label{fig:sub21}
\end{subfigure}
\caption{Two simulations of the SFS of a sample of size $n=100$. We simulate the SFS of the soft and drastic bottleneck structured coalescent for two values of the migration rate (high rate=10, low rate=0.1), while the bottleneck rate is equal to $10$ and the number of islands is equal to $D=4$. 
On the left-hand side, we start our simulation with all the individuals in one island while, on the right-hand side, the initial condition distributes individuals evenly across the four islands.}
\label{fig:SFS1}
\end{figure}

 \begin{figure}
  \centering
  \includegraphics[width=0.5\linewidth]{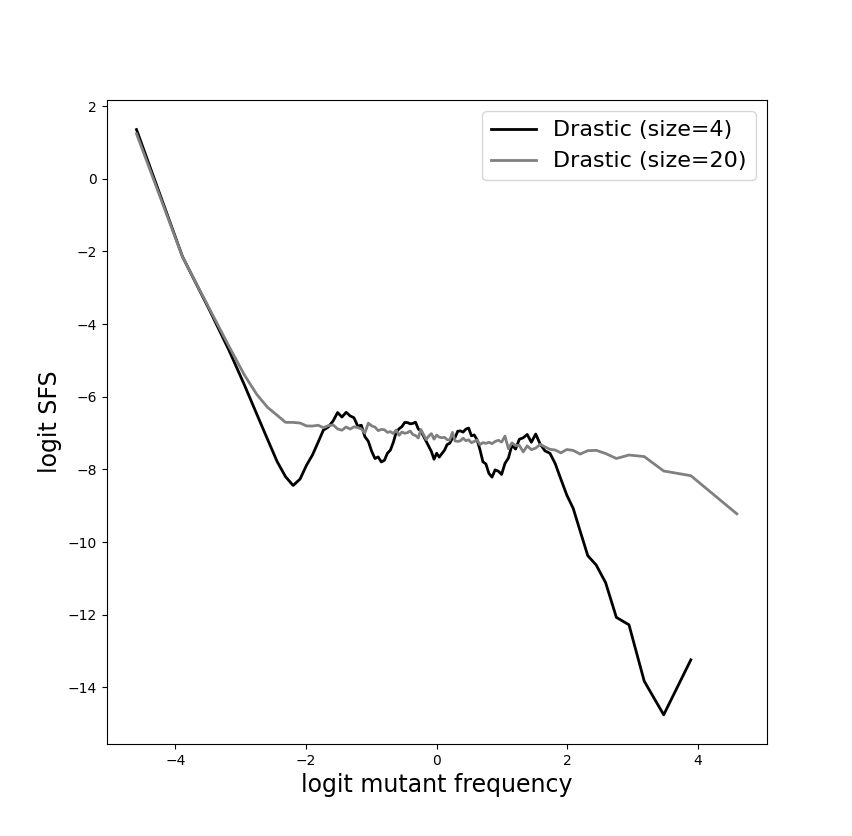}
\caption{Simulations of the SFS of a sample of size $n=100$. We simulate the model with drastic bottleneck without geographical structure (one island) with a very large bottleneck rate of $100$. The black line represents the resulting SFS when the bottleneck is of size $4$, while the grey line the one when the size is equal to $20$.}
\label{fig:SFS3}
\end{figure}

It is a known result \cite{Fu95} that for the Kingman coalescent we have 
\begin{equation*}
\E(M_j(n))=\frac{\vartheta}{j},
\end{equation*}
hence, in this setting, we have a monotone decreasing shape. However, it has been observed that genome data of various species present an excess of low and high frequency variants, resulting in a U-shaped SFS \cite{EBBF15,Freund23, ME20}.
Many $\Lambda$-coalescents and $\Xi$-coalescents predict a U-shaped SFS: they usually have long external branches which lead to a higher number of singletons (mutations appearing in just one individual); on the other hand, high frequency mutations that affect $n-1$ individuals occur on a branch that connect to a subtree of size $n-1$.
Such a branch exists in a Kingman coalescent tree with very low probability, but can be more likely in coalescent with multiple mergers trees which are more unbalanced (see Appendix A.1 in \cite{Freund23}). 
On the left-hand side of Figure~\ref{fig:SFS1} we show that, if we sample individuals all from the same deme and take a small migration rate, we can replicate this simultaneous excess of singletons and high-frequency sites. The corresponding SFS indeed presents a strong U-shape. When we consider a high migration rate, the effect is much weaker. The intuition is that with high migration lineages spread out fast, and hence bottlenecks cannot capture a lot of lineages into large mergers.

On the right-hand side of Figure \ref{fig:SFS1} we repeat the same simulation with a different initial condition: individuals in the sample are evenly distributed across demes. When the migration rate is low, we capture the behaviour of structured coalescents, i.e.\ the SFS is characterised by peaks that appear because many individuals merge within their island before migrating. When the migration rate is high, we lose this shape because the population becomes almost homogeneous.

In Figure \ref{fig:SFS3} we show that we get the 4-peak pattern in the SFS observed in diploid Xi-coalescents derived from generalised Moran models \cite{BCEH16} by using a single island and frequent drastic bottlenecks
in which the population is reduced to four individuals. Since our coalescent is derived from a Wright--Fisher model, diploidy is readily incorporated by doubling the population size inside and outside bottlenecks. Increasing the population size during the bottleneck to 20 results in 19 peaks, each with such a small amplitude that they have effectively disappeared (grey line in Figure \ref{fig:SFS3}).

  Recent data sets obtained from the diploid and highly fecund Atlantic cod \cite{AKHE23} show a clear signal of multiple mergers, but no sign of the intermediate peaks predicted by earlier multiple merger coalescents based on generalised Moran models. Our model is thus a promising candidate for further analysis of such organisms. Furthermore it demonstrates the broader point that multiple merger coalescents in which multiple mergers build up over a few generations (which nevertheless constitute a negligible time increment in the infinite population scaling limit) provide a more realistic basis for data analysis than the more prevalent and mathematically tractable approach in which multiple mergers occur in one generation.

\section*{Acknowledgements}
We thank Adri\'an Gonzalez Casanova for helpful discussions.
Jere Koskela was funded in part by EPSRC research grant EP/V049208/1. Maite Wilke-Berenguer was supported by the Deutsche Forschungsgemeinschaft (DFG, German Research Foundation) under Germany's Excellence Strategy -- The Berlin Mathematics Research Center MATH+ (EXC-2046/1, project ID: 390685689) and acknowledges support from DFG CRC/TRR 388 ``Rough Analysis, Stochastic Dynamics and Related Fields''.

\noindent Data sharing is not applicable to this article as no new data were generated or analyzed.

\appendix
\section{Existence and uniqueness of a strong solution for the limiting drastic bottleneck frequency process}\label{appendixA}
In this section, we prove existence and uniqueness of the limit \eqref{eq:drastic-diffusion} of the allele frequency process of the model with drastic bottlenecks introduced in Section \ref{section:drastic}. We recall that, in the following, $L$ is a probability measure on $\N$ and $F:\N \rightarrow \N$ a function such that $\max_i{c_i}\leq F(n)$ for any $n\in \N$, where $c_i:=\sum_{j\neq i}c_{ij}$. We follow the classic approach by Yamada and Watanabe of proving pathwise uniqueness first, in Proposition \ref{prop:pathwise uniqueness}, and then existence of a weak solution to finally obtain a unique strong solution in Proposition \ref{prop:strong-sol}. 

In order to simplify the notation, we define for $\mathbf{x}\in [0,1]^\demes$,
\begin{align*}
\mathbf{b}(\mathbf{x})&=(b_1(\mathbf{x}),\ldots,b_\demes(\mathbf{x})), \quad b_i(\mathbf{x})=\sum_{j\neq i}\frac{c_{ji}}{\omega_i}(x_j-x_i), \\
\pmb{\sigma}(\mathbf{x})&=(\sigma_{ij}(\mathbf{x}))_{1\leq i,j \leq \demes}, \quad \sigma_{ij}(\mathbf{x})=\1_{\{i=j\}}\sqrt{\omega_i^{-1}x_i(1-x_i)}, \quad \sigma_{i}(x_i):=\sigma_{ii}(\mathbf{x}),
\end{align*}
and the function $\mathbf{g}:[0,1]^\demes \times \mathcal{S}\rightarrow [0,1]^\demes$ as $\mathbf{g}(\mathbf{x},  \tau, a,  u) := (g_1(\mathbf{x},  \tau, a,  u),\ldots,g_D(\mathbf{x},  \tau, a,  u))$ with
\begin{equation}\label{eq:definition_c}
g_i(\mathbf{x},\tau,a,u)=\frac{1}{\size(\tau)}\Big[\sum_{j=1}^{\demes}\sum_{m=1}^{\infty}a_{jm}(\tau)\1_{\{u_{jm}\leq x_j\}}\Big]-x_i,
\end{equation}
where $\mathcal{S}:=\N_0 \times \N_0^{\demes \times \N} \times [0,1]^{\demes \times \N}$. Using $\mathbf{X}(t):=(X_1(t), \ldots,  X_D(t))$, $t \geq 0$, we can rewrite the SDE in \eqref{eq:drastic-diffusion} as

\begin{equation}\label{eq:drastic-diffusion-simply}
d\mathbf{X}(t)=\mathbf{b}\left(\mathbf{X}(t)\right)dt+\pmb{\sigma}\left(\mathbf{X}(t)\right)d\mathbf{W}(t)+\int_\mathcal{S} \mathbf{g}\left(\mathbf{X}(t^-),\tau,a,u\right)\,\mathbf{N}(dt,d\tau,da,du),	
\end{equation}
where $\mathbf{W}$ is a $D$-dimensional Brownian motion and $\mathbf{N}=(\mathcal{N}_1,\ldots ,\mathcal{N}_D)$ are independent Poisson random measures on $(0,\infty)\times \mathcal{S}$ with intensity $\gamma ds \otimes \mathfrak{A}^\i(d\tau,da)\otimes du$ for $i=1,\ldots ,D$, , where $du$ is the Lebesgue measure on $[0,1]^{\demes \times \N}$ and $\mathfrak{A}^\i$  is the semidirect product of $L$ and the probability kernel $A^\i(\cdot,\cdot)$. Recall that $A^\i(\tau,\cdot)$ was the distribution of the matrix $\mathbf{A}^\i(\tau)\in  \N_0^{\demes\times \N}$ of the final family sizes after a bottleneck affecting deme $i$ for $\tau$ generations with population size determined by $\size$, see page \pageref{object:familysizesendofbottleneck}. 
\begin{proposition}\label{prop:pathwise uniqueness}
Let $L$ be a probability measure on $\N$, and $F:\N \rightarrow \N$ a function such that $\max_i{c_i}\leq F(n)$ for any $n\in \N$, where $c_i:=\sum_{j\neq i}c_{ij}$. Pathwise uniqueness holds for equation \eqref{eq:drastic-diffusion}.
\end{proposition}
\begin{proof}
The drift and the diffusion coefficients satisfy the conditions of the classical result by Yamada and Watanabe \cite{YW71}. Indeed
\begin{enumerate}
\item[(i)] there exists a constant $K_{\sigma}$ such that, for every $x,z\in [0,1]$ and $i=1,\ldots,\demes$,
\begin{equation*}
|\sigma_i(x)-\sigma_i(z)|\leq K_{\sigma}\sqrt{\abs{x-z}},
\end{equation*}
\item[(ii)]	there exists a constant $K_{b}$ such that, for every $\mathbf{x},\mathbf{z}\in [0,1]^\demes$ and $i=1,\ldots,\demes$,
\begin{equation*}
|b_i(\mathbf{z})-b_i(\mathbf{x})|\leq K_{b}\norm{\mathbf{x}-\mathbf{z}},
\end{equation*}
\end{enumerate}
where with $\norm{\cdot}$ we indicate the Euclidean norm. In addition, in Lemma \ref{lemma:cLip} below, we prove that
\begin{enumerate}
\item[(iii)] there exists a constant $K_{c}$ such that, for every $\mathbf{x},\mathbf{z}\in [0,1]^\demes$ and $i=1,\ldots,\demes$,
\begin{equation*}
\int_\mathcal{S}|g_i(\mathbf{z},\tau,a,u)-g_i(\mathbf{x},\tau,a,u)|\,\mathfrak{A}^\i(d\tau,da)\,du\leq K_{c}\norm{\mathbf{x}-\mathbf{z}}.
\end{equation*}	
\end{enumerate}

We now follow the proof of Theorem 1 in \cite{YW71} considering the sequence $a_0=1>a_1>a_2>\ldots>a_k>\ldots\rightarrow 0$ defined by
\begin{equation*}
\int_{a_1}^{a_0}\frac{1}{r}dr=1,\,\int_{a_2}^{a_1}\frac{1}{r}dr=2,\,\ldots,\,\int_{a_k}^{a_{k-1}}\frac{1}{r}dr=k,\, \ldots.
\end{equation*}
Then, for all $k\geq 1$, there exists a twice continuously differentiable function $\varphi_k$ on $[0,\infty)$  such that $\varphi_k(0)=0$ and
\begin{equation*}
\varphi'_k(r) =   \left\{ \begin{array}{ll}
0 & 0\leq r\leq a_k,\\
\text{between }0 \text{ and }1 & a_k<r<a_{k-1},\\
1& r\geq a_{k-1}, 
\end{array}
\right.
\end{equation*}
 and 
\begin{equation*}
\varphi''_k(r) =   \left\{ \begin{array}{ll}
0 & 0\leq r\leq a_k,\\
\text{between }0 \text{ and }\frac{1}{k} \frac{1}{r} & a_k<r<a_{k-1},\\
0& r\geq a_{k-1}. 
\end{array}
\right.
\end{equation*}
We extend these functions symmetrically to $(-\infty,\infty)$ so that $\varphi_k(r)=\varphi_k(|r|)$.
We have $\varphi_k(r)\uparrow |r|$ as $k\rightarrow \infty$ for each $r\in \R$. Now let $\mathbf{X}$ and $\mathbf{X}'$ be two solutions of \eqref{eq:drastic-diffusion-simply} on the same probability space and with respect to the same Brownian motion and Poisson processes, with $\mathbf{X}(0)=\mathbf{X}'(0)$.
Then 
\begin{equation*}
\begin{split}
X_i(t)-X'_i(t)=&\int_{0}^t\big[\sigma_i(X_i(s))-\sigma_i(X'_i(s))\big] dW_i(s)+\int_0^t\big[b_i(\mathbf{X}(s))-b_i(\mathbf{X}'(s))\big]	ds\\&+\int_0^t\int_\mathcal{S} \big[g_i(\mathbf{X}(s),\tau,a,u)-g_i(\mathbf{X}'(s),\tau,a,u)\big]\,\mathbf{N}(ds,d\tau,da,du).
\end{split}
\end{equation*}
By It\=o's formula (see for example \cite[Theorem 4.4.7]{Applebaum_2009}), for any $k \in \N$, $ i=1, \ldots, D$
\begin{equation*}
\begin{split}
\E\big[\varphi_k(X_i(t)-X'_i(t))\big]   &=\E\left[\int_{0}^t\varphi'_k(X_i(s)-X'_i(s))\big[b_i(\mathbf{X}(s))-b_i(\mathbf{X}'(s))\big]	ds\right]\\
    &\quad +\E\left[\frac{1}{2}\int_0^t\varphi''_k(X_i(s)-X'_i(s))\big[\sigma_i(X_i(s))-\sigma_i(X'_i(s))\big]^2ds\right]\\
    & \quad +\E\Bigg[\int_0^t\int_{\mathcal S} \Big[\varphi_k\left(X_i(s)-X'_i(s)+g_i(\mathbf{X}(s),\tau,a,u)-g_i(\mathbf{X}'(s),\tau,a,u)\right)\\
    &\quad \qquad \qquad \qquad -\varphi_k\left(X_i(s)-X'_i(s)\right)\Big]\,\mathfrak{A}^\i(d\tau,da)\,du\,ds\Bigg].
\end{split}
\end{equation*}
We bound the first two terms  as in the proof  of Theorem 1 in \cite{YW71}. Since $\varphi'_k$ is uniformly bounded, by (ii) 
we have
\begin{equation*}
\begin{split}
\left|\E\left[\int_{0}^t\varphi'_k(X_i(s)-X'_i(s))\big[b_i(\mathbf{X}(s))-b_i(\mathbf{X}'(s))\big]ds\right]\right|\leq K_b\int_{0}^t\E[\norm{\mathbf{X}(s)-\mathbf{X}'(s)}]\,ds.
\end{split}
\end{equation*}
For the second term we can estimate for every $i=1,\ldots, D$ and uniformly in $\omega$ 
\begin{align*}
\Big|\frac{1}{2}\int_0^t    & \varphi''_k(X_i(s)-X'_i(s))\big[\sigma_i(X_i(s))-\sigma_i(X'_i(s))\big]^2ds\Big|\\
            &\qquad\leq \frac{K^2_{\sigma}}{2}\int_0^t\varphi''_k(X_i(s)-X'_i(s))\vert  X_i(s)- X_i'(s)\vert ds\\
            &\qquad\leq \frac{K^2_{\sigma}}{2} t\cdot \max_{a_k\leq \abs{x}\leq a_{k-1}}[\varphi''_k(x)\,\vert x\vert]\leq \frac{K^2_{\sigma}}{2}\,t\cdot \frac{1}{k}\rightarrow 0, \,\,\,\,\,\, \text{as }k\rightarrow \infty
\end{align*}
for any $t \geq 0$. Using (iii) and the fact that $\varphi_k$ is Lipschitz with Lipschitz constant smaller than $1$, for the jump term we have, for all $ k\in \N$, $i=1, \ldots, D$,
\begin{equation*}
\begin{split}
&\E\Bigg[\int_0^t\int_{\mathcal S} \Big|\varphi_k\big(X_i(s)-X'_i(s)+g_i(\mathbf{X}(s),\tau,a,u)-g_i(\mathbf{X}'(s),\tau,a,u)\big)\\
&\qquad\qquad\quad -\varphi_k\big(X_i(s)-X'_i(s)\big)\Big|\,\mathfrak{A}^\i(d\tau,da)\,du\,ds\Bigg]\\&\qquad\leq \E\left[\int_0^t\int_{\mathcal S} |g_i(\mathbf{X}(s),\tau,a,u)-g_i(\mathbf{X}'(s),\tau,a,u)|\,\mathfrak{A}^\i(d\tau,da)\,du\,ds\right]\\
&\qquad\leq K_c\int_0^t\E[\norm{\mathbf{X}(s)-\mathbf{X}'(s)}]ds.
\end{split}
\end{equation*}
Combining these bounds, and taking the limit as $k\to \infty$, we have, by monotone convergence, for some constant $K$,
\begin{equation}\label{eq:Gronwall}
	\E[\norm{\mathbf{X}(t)-\mathbf{X}'(t)}]\leq K\int_0^t\E[\norm{\mathbf{X}(s)-\mathbf{X}'(s)}]ds,
\end{equation}
which, by Gr\"onwall's inequality, implies $\E[\norm{\mathbf{X}(t)-\mathbf{X}'(t)}]= 0$ $\forall t\geq 0$ and, since our processes are right continuous, we can conclude $\P\left(\forall t\geq 0,\,\mathbf{X}(t)= \mathbf{X}'(t)\right)=1$.
\end{proof}
Now we turn to the proof of the estimate (iii) fur the jump-part,  that we used in the previous proof.  
\begin{lemma}\label{lemma:cLip} 
Let $F:\N \rightarrow \N$ be a function such that $\max_i{c_i}\leq F(n)$ for any $n\in \N$. Given the function $\mathbf{g}:[0,1]^\demes \times \mathcal{S}\rightarrow [0,1]^\demes$ defined in \eqref{eq:definition_c},
there exists a constant $K_c$ such that for any $\mathbf{x},\mathbf{z}\in [0,1]^\demes$
\begin{equation}\label{eq:cLip}
\int_\mathcal{S}|g_i(\mathbf{z},\tau,a,u)-g_i(\mathbf{x},\tau,a,u)|\,\mathfrak{A}^\i(d\tau,da)\,du\leq K_{c}\norm{\mathbf{x}-\mathbf{z}},
\end{equation}	
for all $i=1,\ldots,\demes$.
\end{lemma}
\begin{proof}
We can rewrite the left-hand side of \eqref{eq:cLip} as
\begin{equation*}
\begin{split}
&\int_\mathcal{S}|g_i(\mathbf{z},\tau,a,u)-g_i(\mathbf{x},\tau,a,u)|\,\mathfrak{A}^\i(d\tau,da)\,du\\&=\E\left[\,\left|\frac{1}{\size(G)}\left(\sum_{j=1}^{\demes}\sum_{m=1}^{\infty}A^\i_{jm}(\1_{\{\mathcal{U}_{jm}\leq x_j\}}-\1_{\{\mathcal{U}_{jm}\leq z_j\}})\right)-(x_i-z_i)\right|\,\right],\end{split}
\end{equation*}
where $G$ is a random variable with law $L$, $\mathbf{A}^\i=(A^\i_{jm})_{j=1, 
 \ldots,   \demes, m   \in \N}$ is a random matrix with conditional distribution $A^\i(G,\cdot)$,   and the $\mathcal{U}_{jm}$ are independent uniform random variables on $[0,1]$, also  independent of $G$. 
On the event $G=\tau$, $A^\i_{jm}=a^\i_{jm}$ for $i\in \N$, we get 
\begin{equation*}
\begin{split}
&\E\left[\,\left|\frac{1}{\size(\tau)}\left(\sum_{j=1}^{\demes}\sum_{m=1}^{\infty}a^\i_{jm}(\1_{\{\mathcal{U}_{jm}\leq x_j\}}-\1_{\{\mathcal{U}_{jm}\leq z_j\}})\right)-(x_i-z_i)\right|\,\right]\\& \hspace{0.5cm}\leq \frac{1}{\size(\tau)}\left[\sum_{j=1}^{\demes}\sum_{l=1}^{\infty}a^\i_{jm}\abs{x_j-z_j} \right]	+\abs{x_i-z_i}.
\end{split}
\end{equation*}
Using the independence of $\mathcal{U}_{jm}$ with respect to the other variables, we have
\begin{equation*}
\begin{split}
&\E\left[\,\left|\frac{1}{\size(G)}\left(\sum_{j=1}^{\demes}\sum_{m=1}^{\infty}A^\i_{jm}(\1_{\{\mathcal{U}_{jm}\leq x_j\}}-\1_{\{\mathcal{U}_{jm}\leq z_j\}})\right)-(x_i-z_i)\right|\,\right]\\
&\leq \E \left[\frac{1}{\size(G)}\left(\sum_{j=1}^{\demes}\sum_{m=1}^{\infty}A^\i_{jm}\abs{x_j-z_j} \right)	+\abs{x_i-z_i}\right] \\
&\leq \sum_{j=1}^{\demes}\abs{x_j-z_j} \,\E\left[\frac{1}{\size(G)}\left(\sum_{m=1}^{\infty}A^\i_{jm}\right)\right]+\abs{x_i-z_i}\leq K_c\norm{\mathbf{x}-\mathbf{z}}.
\end{split}
\end{equation*}
\end{proof}

\begin{proposition}\label{prop:strong-sol}Let $L$ be a probability measure on $\N$, and $F:\N \rightarrow \N$ a function such that $\max_i{c_i}\leq F(n)$ for any $n\in \N$, where $c_i:=\sum_{j\neq i}c_{ij}$.
There exists a unique strong solution of the SDE system \eqref{eq:drastic-diffusion}.	
\end{proposition}

\begin{proof}
Consider $f \in C^3([0,1]^\demes)$ and let $Df=(\frac{\partial f}{\partial x_1},\ldots,\frac{\partial f}{\partial x_{\demes}})$. Consider the infinitesimal generator of (\ref{eq:drastic-diffusion-simply}):
\begin{equation*}
\begin{split}
\mathcal{L}f(\mathbf{x}):= {}&\Big\langle Df(\mathbf{x}),\mathbf{b}(\mathbf{x})+\int_\mathcal{S}\mathbf{g}(\mathbf{x,r})\,\nu(d\mathbf{r})\Big\rangle + \frac{1}{2} tr(\pmb{\sigma}(\mathbf{x})\pmb{\sigma}(\mathbf{x})^TD^2f(\mathbf{x}))\\&+\int_\mathcal{S}[f(\mathbf{x}+\mathbf{g}(\mathbf{x,r}))-f(\mathbf{x})-\langle Df(\mathbf{x}),\mathbf{g}(\mathbf{x,r})\rangle]\,\nu(d\mathbf{r}),
\end{split}
\end{equation*}
where $\mathbf{r}=(\tau,a,u)\in\mathcal{S}$ and $\nu=\mathfrak{A}^\i(d\tau,da)\otimes du$. We wish to show that \eqref{eq:drastic-diffusion} has a unique solution. Following the steps of the proof of Theorem 2.8 in \cite{XZ_2019}, we define the following measure on the Borel sets $\mathcal{B}([0,1]^\demes)$. For any $\mathbf{x} \in [0,1]^\demes$, $M(\mathbf{x},B):=\nu\{\mathbf{r}\in \mathcal{S}\,:\, \mathbf{g}(\mathbf{x,r})\in B\}$. The generator $\mathcal{L}$ can then be rewritten as:
\begin{equation*}
\begin{split}
\mathcal{L}f(\mathbf{x}):={}& \Big\langle Df(\mathbf{x}),\mathbf{b}(\mathbf{x})+\int_{[0,1]^\demes}\mathbf{y}\,M(\mathbf{x},d\mathbf{y})\Big\rangle + \frac{1}{2} tr(\pmb{\sigma}(\mathbf{x})\pmb{\sigma}(\mathbf{x})^TD^2f(\mathbf{x}))\\&+\int_{[0,1]^\demes}[f(\mathbf{x+y})-f(\mathbf{x})-\langle Df(\mathbf{x}),\mathbf{y}\rangle] M(\mathbf{x},d\mathbf{y}).
\end{split}
\end{equation*}
We can use Theorem 2.2 in \cite{Stroock_1975} to show that the martingale problem associated to $\mathcal{L}$ has a solution, since all the assumptions are satisfied by our generator. 
Since we are working in a compact domain, boundedness assumptions are clearly satisfied.
It remains to check that 
\begin{equation*}
\int_{[0,1]^\demes}\frac{|\mathbf{y}|^2}{1+|\mathbf{y}|^2}\varphi(\mathbf{y})M(\mathbf{x},d\mathbf{y})	
\end{equation*}                                 
is continuous for all $\varphi \in C_b(\R^\demes)$. Now we observe that
\begin{equation*}
\int_{[0,1]^\demes}\frac{|\mathbf{y}|^2}{1+|\mathbf{y}|^2}\varphi(\mathbf{y})M(\mathbf{x},d\mathbf{y})	= \int_\mathcal{S} \frac{|\mathbf{g}(\mathbf{x,r})|^2}{1+|\mathbf{g}(\mathbf{x,r})|^2}\varphi(\mathbf{g}(\mathbf{x,r}))\nu(d\mathbf{r}),
\end{equation*}
which is indeed continuous in $\mathbf{x}$.
Now, by Theorem 2.3 in \cite{Kurtz_2011}, we know that having a solution of the martingale problem associated with $\mathcal{L}$ is equivalent to having a weak solution of (\ref{eq:drastic-diffusion}).
The existence of a weak solution and the pathwise uniqueness proved in Proposition \ref{prop:pathwise uniqueness} imply the existence of a strong solution (see Theorem 1.2 in \cite{BLP_2015}).
\end{proof}

\begin{proposition}\label{prop:Feller}
Let $L$ be a probability measure on $\N$, and $F:\N \rightarrow \N$ a function such that $\max_i{c_i}\leq F(n)$ for any $n\in \N$, where $c_i:=\sum_{j\neq i}c_{ij}$. The unique strong solution $(\mathbf{X}(t))_{t\geq 0}$ of \eqref{eq:drastic-diffusion} is Feller continuous.	
\end{proposition}
\begin{proof}
Let $\mathbf{X}$ and $\mathbf{X}'$ be two solutions of \eqref{eq:drastic-diffusion-simply} on the same probability space started in $\mathbf{x}$ and $\mathbf{x}'$ respectively. 
We have to prove that, for any continuous bounded function $f$ on $[0,1]$,
\begin{equation}\label{eq:feller}
\E[f(\mathbf{X}'(t))] \rightarrow \E[f(\mathbf{X}(t))]	 \quad \text{as } \norm{\mathbf{x}-\mathbf{x}'}\to 0, \,  \text{for all } t\geq 0.
\end{equation}
We may assume that the two solutions are driven by the same Brownian motion and Poisson process, i.e.\ that they satisfy 
\begin{equation*}
\mathbf{X}(t)=\mathbf{x}+\int_{0}^t \pmb{\sigma}(\mathbf{X}(s)) d\mathbf{W}_s+\int_0^t\mathbf{b}(\mathbf{X}(s))	ds+\int_0^t\int_\mathcal{S} \mathbf{g}(\mathbf{X}(s),\tau,a,u)\,\mathbf{N}(ds,d\tau,da,du),
\end{equation*}
\begin{equation*}
\mathbf{X}'(t)=\mathbf{x}'+\int_{0}^t \pmb{\sigma}(\mathbf{X}'(s)) d\mathbf{W}_s+\int_0^t\mathbf{b}(\mathbf{X}'(s))	ds+\int_0^t\int_\mathcal{S} \mathbf{g}(\mathbf{X}'(s),\tau,a,u)\,\mathbf{N}(ds,d\tau,da,du).
\end{equation*}
The estimate \eqref{eq:Gronwall} from the proof of Proposition \ref{prop:pathwise uniqueness} and Gronwall's inequality in this case imply 
\begin{align*}
    \E[\norm{\mathbf{X}(t)-\mathbf{X}'(t)}]\leq \norm{\mathbf{x}-\mathbf{x}'}\exp(Kt)
\end{align*}
and thus 
\begin{equation*}
\lim_{\norm{\mathbf{x}-\mathbf{x}'}\to 0}\E\left[\norm{\mathbf{X}(t)-\mathbf{X}'(t)}\right]=0 
\end{equation*}
for all $t \geq 0$. This in turn implies convergence in probability and we obtain
Equation \eqref{eq:feller} from the bounded convergence theorem for convergence in probability.
\end{proof}

\section{The structured Wright--Fisher model and Kingman coalescent during a soft bottleneck}\label{appendixB} 
In this section, we consider the setting of Section \ref{section:soft} and investigate the behaviour of the frequency and ancestral processes during one soft bottleneck length $l^N$ affecting deme $i$. In particular, consider a population divided in $\demes$ demes evolving in discrete generations $g\in\{0,1,\ldots ,l^N\}$ with $l^N \to \infty$ and $l^N/N\to 0$ as $N \to \infty$.  The size of every deme $j\neq i$ is constant and equal to $N_j=\omega_jN$, while size of the $i$-th subpopulation is equal to $\presize(g)$, where $\presize:[0,l^N]\to \N$ is such that 
\begin{equation}\label{eq:condition-presize}
\lim_{N\to \infty} \frac{\presize(\lfloor l^Nt\rfloor)}{l^N}=\varphi(t) \hspace{2em} \forall \,t\in [0,1],
\end{equation}
for some continuous function $\varphi:[0,1]\rightarrow (0,\infty)$. The evolution follows the mechanism described in Section \ref{sec:model}: reproduction follows a Wright--Fisher model within each deme, and at each time step a fixes number $c_{ij}$ of individuals migrates from deme $i$ to deme $j$ with
\begin{equation*}
c_i:=\sum_{j\neq j}c_{ij}=\sum_{i\neq j}c_{ji}.
\end{equation*} 
Each individual carries one of two possible alleles $\{a,A\}$ and passes its type to the offspring. We consider the $\N_0^\demes$-valued process $\{\pmb{\mathcal{X}}^{\i,N}(g)\}_{g\in\{0,\ldots ,l^N\}}$, where
\begin{equation*}
\begin{split}
&\mathcal{X}^{\i,N}_j(g)=\text{number of type } a\text{ individuals in deme $j$ at time }g.
\end{split}
\end{equation*}
Since the size of one deme in not constant, $\{\pmb{\mathcal{X}}^{\i,N}(g)\}_{g\in\{0,\ldots ,l^N\}}$ is a time-inhomogeneous Markov chain with 
\begin{equation*}
\mathcal{X}^{\i,N}_{j}(g)=V^{\i,N}_{jj}(g)+
	 \sum_{l \neq j} V^{\i,N}_{lj}(g),
\end{equation*}
where $V^{\i,N}_{jj}(g)$ is the number of type-$a$ individuals in deme $j$ at time $g$ whose ancestor was in deme $j$ at time $g-1$, while $V^{\i,N}_{lj}(g)$ is the number of type-$a$ individuals in deme $j$ at time $g$ whose ancestor was in deme $l$ at time $g-1$. The conditional distributions of these variables are
\begin{equation*}
\begin{split}
&\mathcal{L}(V^{\i,N}_{jj}(g+1)|\pmb{\mathcal{X}}^{\i,N}(g))=\text{Bin}\left(N_j-c_j,\frac{\mathcal{X}^{\i,N}_j(g)}{N_j}\right),\\
&\mathcal{L}(V^{\i,N}_{lj}(g+1)|\pmb{\mathcal{X}}^{\i,N}(g))=\text{Bin}\left(c_{lj},\frac{\mathcal{X}^{\i,N}_l(g)}{N_l}\right),\\
&\mathcal{L}(V^{\i,N}_{ij}(g+1)|\pmb{\mathcal{X}}^{\i,N}(g))=\text{Bin}\left(c_{ij},\frac{\mathcal{X}^{\i,N}_i(g)}{\presize(g)}\right),\\
&\mathcal{L}(V^{\i,N}_{ii}(g+1)|\pmb{\mathcal{X}}^{\i,N}(g))=\text{Bin}\left(\presize(g+1)-c_i,\frac{\mathcal{X}^{\i,N}_{i}(g)}{\presize(g)}\right),\\
&\mathcal{L}(V^{\i,N}_{li}(g+1)|\pmb{\mathcal{X}}^{\i,N}(g))=\text{Bin}\left(c_{lj},\frac{\mathcal{X}^{\i,N}_l(g)}{N_l}\right).
\end{split}
\end{equation*}
The frequency of type-$a$ individuals in the demes is given by the $[0,1]^\demes$-valued Markov chain $\{\tilde{\mathbf{X}}^{\i,N}(g)\}_{g\in\{0,\ldots ,l^N\}}$ with 
\begin{equation}\label{eq:appendix-prelimiting-frequency}
\tilde{X}^{\i,N}_i(g)=\frac{\mathcal{X}^{\i,N}_i(g)}{\presize(g)} \,\, \text{ and } \,\, \tilde{X}^{\i,N}_j(g)=\frac{\mathcal{X}^{\i,N}_j(g)}{N_j} \,\, \text{ for }\,\, j\neq i.
\end{equation}
We now prove a convergence result for this allele-frequency process. Since the sizes of all but one deme are of order $N\gg l^N$, in the timescale $l^N$, we will eventually observe fluctuations only in the $i$-th subpopulation, while, in all the others, the allele-frequency remains constant.
\begin{proposition}\label{prop:appendix-frequency-convergence}
Fix a sequence $l^N$ with $l^N\rightarrow \infty$ and $l^N/N \rightarrow 0$ as $N\rightarrow \infty $, and a sequence of functions $\presize$ satisfying \eqref{eq:condition-presize}. Let $\{\tilde{\mathbf{X}}^{\i,N}(g)\}_{g\in\{0,\ldots ,l^N\}}$ be the frequency process as defined in \eqref{eq:appendix-prelimiting-frequency}, with $\tilde{\mathbf{X}}^{\i,N}(0)=\mathbf{x}\in [0,1]^\demes$. Then $\{\tilde{\mathbf{X}}^{\i,N}(\lfloor l^Nt\rfloor)\}_{t\in[0,1]}	$ converges weakly in distribution, as $N \to \infty$, to $(\tilde{\mathbf{X}}^\i(\size(t)))_{t\in [0,1]}$, where
\begin{equation*}
\size(t)=\int_0^t\frac{1}{\varphi(s)}\,ds,
\end{equation*}
and $(\tilde{\mathbf{X}}^\i(t))_{t\in [0,1]}$ is the unique strong solution of 
\begin{equation*}
\begin{split}
d\tilde{X}_i^\i(t)=&\sum_{j\neq i}c_{ji}(\tilde{X}^\i_j(t)-\tilde{X}_i^\i(t))dt+\sqrt{\tilde{X}_i^\i(t)(1-\tilde{X}_i^\i(t))}dB_i(t),\\
d\tilde{X}_j^\i(t)=&\,0 \,\, \text{ for } j\neq i,
\end{split}
\end{equation*}	
where $(B_i(t))_{t \in [0,1]}$ is a Brownian  motion and $\tilde{\mathbf{X}}^\i(0)=\mathbf{x}$.	
\end{proposition}
\begin{proof}
As we have a time-inhomogeneous process, we consider the homogeneous extension
\begin{equation*}
\{\tilde{\mathbf{X}}^{\i,N}(\lfloor l^Nt\rfloor),\lfloor l^Nt\rfloor\}_{t\geq 0},
\end{equation*}
whose generator is
\begin{equation*}
\begin{split}
\tilde{\mathcal{A}}^{\i,N}f(\mathbf{x},t)={}&l^N\E_{\mathbf{x}}\left[f\left(\tilde{\mathbf{X}}^{\i,N}(\lfloor l^Nt\rfloor+1),\frac{\lfloor l^Nt\rfloor+1}{l^N}\right)-f\left(\mathbf{x},\frac{\lfloor l^Nt\rfloor+1}{l^N}\right)\right]\\&+l^N\left[f\left(\mathbf{x},\frac{\lfloor l^Nt\rfloor+1}{N}\right)-f\left(\mathbf{x},\frac{\lfloor l^Nt\rfloor}{l^N}\right)\right],
\end{split}
\end{equation*}
for any $f\in C^3([0,1]^\demes)$. By Taylor expansion, we can prove that the first term of $\tilde{\mathcal{A}}^{\i,N}f(\mathbf{x},t)$ converges, as $N \to \infty$, to 
\begin{equation*}
\frac{\partial f}{\partial x_i}(\mathbf{x},t)\frac{\sum_{j\neq i}c_{ji}(x_j-x_i)}{\varphi(t)}+\frac{\partial ^2f}{\partial x_i^2}(\mathbf{x},t)\frac{x_i(1-x_i)}{2\varphi(t)},
\end{equation*}
while for the second term we have
\begin{equation*}
l^N\left[f\left(\mathbf{x},\frac{\lfloor l^Nt\rfloor+1}{l^N}\right)-f\left(\mathbf{x},\frac{\lfloor l^Nt\rfloor}{l^N}\right)\right]=\frac{ f\left(\mathbf{x},\lfloor t\rfloor+1/l^N\right)-f\left(\mathbf{x},\lfloor t\rfloor\right)}{1/l^N}\xrightarrow{N\rightarrow \infty}\frac{\partial f}{\partial t}(\mathbf{x},t).
\end{equation*}
Hence we get
\begin{equation}\label{eq:appendix-limiting-generator}
\lim_{N\to \infty}\tilde{\mathcal{A}}^{\i,N}f(\mathbf{x},t)=\frac{\partial f}{\partial x_i}(\mathbf{x},t)\frac{\sum_{j\neq i}\omega_j\nu_{ji}(x_j-x_i)}{\varphi(t)}+\frac{\partial ^2f}{\partial x_i^2}(\mathbf{x},t)\frac{x_i(1-x_i)}{2\varphi(t)}+\frac{\partial f}{\partial t}(\mathbf{x},t).
\end{equation}
The limit in \eqref{eq:appendix-limiting-generator} is indeed the generator of the process $(\tilde{\mathbf{X}}^\i(\size(t)))_{t\in [0,1]}$. As our limit is a Feller processes, we conclude that the convergence of generators implies the weak convergence \cite[Theorem 19.25 and 19.28]{Kallenberg}.
\end{proof}
We now consider the same evolution backward in time, i.e.\ time $g$ now corresponds to generation $l^N-g$, and build the discrete \emph{ancestral line process} of this model taking values in the space of \emph{marked} partitions 
\begin{equation*}
	\mathcal{P}_n^{\demes}:=\{\xi=(\pi,\vec{u})\,|\, \pi\in \mathcal{P}_n, \vec{u}\in\{1,\ldots,\demes\}^{|\xi|}\}.
\end{equation*}
\begin{definition}\label{def:pi-tilda}
We start with a sample of size $n$ at time $0$, i.e\ our initial condition is of the form $\xi=(\pi,\vec{u})$ with $\pi=\{\{1\},\{2\},\ldots ,\{n\}\}$, and then we trace their ancestries through the equivalence relation 
\begin{equation*}
i\sim_g	j \quad \quad \text{if individuals $i$ and $j$ find a common ancestor before time $g$,}
\end{equation*}
the mark of each block corresponds to the deme those ancestors are in at that time. We call the process constructed as such \emph{ancestral lime process of the structure Wright--Fisher model with $(i)$-varying population size} and denote it by  $\{\tilde{\Pi}^{\i,N}(g)\}_{g \in \{0,\ldots ,l^N\}}$.
\end{definition}
As for the frequency process, the intuition is that, in the limit, we only see changes in the small $i$-th subpopulation. Indeed, the limiting ancestral process will show only binary mergers in deme $i$ and single migrations from deme $i$ to the others.
\begin{proposition}\label{prop:appendix-coalescent-convergence}
Fix a sequence $l^N$ with $l^N\rightarrow \infty$ and $l^N/N \rightarrow 0$ as $N\rightarrow \infty $, and a sequence of functions $\presize$ satisfying \eqref{eq:condition-presize}. Let $\{\tilde{\Pi}^{\i,N}(g)\}_{g \in \{0,\ldots,l^N\}}$ be the process defined in \ref{def:pi-tilda}, then we have 
\begin{equation*}
\{\tilde{\Pi}^{\i,N}(\lfloor Nt\rfloor)\}_{t\in [0,1]} \xRightarrow[N\rightarrow \infty]{} \{\tilde{\Pi}^{\i}(F(t))\}_{t\in [0,1]}\,,
\end{equation*}
where $(\tilde{\Pi}^{\i}(t))_{t\in[0,1]}$ is a Markov process on $\mathcal{P}_k^{\demes}$ with rate matrix $\tilde{Q}^{\text{\tiny{($i$)}}}=\{\tilde{q}_{\xi\eta}^{\text{\tiny{($i$)}}}\}$ such that

\begin{equation*}\arraycolsep=3pt\def\arraystretch{1.9} \tilde{q}_{\xi\eta}^{\text{\tiny{($i$)}}} =   \left\{ \begin{array}{lll}
         1 & \text{if} \hspace{0.2cm}&  \xi \succi \eta, \\
        c_{ji}& \text{if} \hspace{0.2cm}& \xi \simi \eta,
            \end{array} \right. 
        \end{equation*}
and 
\begin{equation*}
\size(t)=\int_0^t\frac{1}{\varphi(s)}\,ds.
\end{equation*}
\end{proposition}
\begin{proof}
\noindent The proof follows as in Proposition \ref{prop:appendix-frequency-convergence} by convergence of generators.
\end{proof}

\bibliographystyle{alpha}
\bibliography{BibBottleneck.bib}

\end{document}